\documentclass[12pt]{amsart}
\usepackage[margin=1.15in]{geometry}

\usepackage{amssymb,amsthm,amsmath,latexsym,stmaryrd}
\usepackage[all]{xy}

\newcommand{\complex}{{\mathbb C}}
\newcommand{\naturals}{{\mathbb N}}
\newcommand{\reals}{{\mathbb R}}

\newcommand{\integers}{{\mathbb Z}}

\newcommand{\cala}{{\mathcal A}}

\newcommand{\calf}{{\mathcal F}}

\newcommand{\cals}{{\mathcal S}}

\newcommand{\calv}{{\mathcal V}}

\newcommand{\calw}{{\mathcal W}}

\theoremstyle{plain}
        \newtheorem{theorem}{Theorem}[section]
        \newtheorem{lemma}[theorem]{Lemma}
        \newtheorem{remark}[theorem]{Remark}
        \newtheorem{proposition}[theorem]{Proposition}

        \newtheorem{example}[theorem]{Example}
\theoremstyle{definition}
        \newtheorem{definition}[theorem]{Definition}

\newcommand{\op}{\operatorname{Op}}
\title{Dunkl operator and quantization of $\mathbb{Z}_2$-singularity}
\author{Gilles Halbout}
\address{Institut de Math\'ematiques et de Mod\'elisation de Montpellier,
Universit\'e de Montpellier 2, CC5149, Place Eug\`ene Bataillon,
F-34095 Montpellier CEDEX 5, France}
\email{halbout@math.univ-montp2.fr}

\author{Xiang Tang}
\address{Department of Mathematics, Washington University, St. Louis, Missouri, USA, 63130}
\email{xtang@math.wustl.edu}

\begin{document}
\begin{abstract}
Let $(X,\omega)$ be a symplectic orbifold which is locally like the
quotient of a $\mathbb{Z}_2$ action on $\reals^n$. Let
$A^{((\hbar))}_X$ be a deformation quantization of $X$ constructed
via the standard Fedosov method with characteristic class being
$\omega$. In this paper, we construct a universal deformation of the
algebra $A^{((\hbar))}_X$ parametrized by codimension 2 components
of the associated inertia orbifold $\widetilde{X}$. This partially
confirms
 a conjecture of Dolgushev and Etingof (see \cite{D-E}) in the case
of $\mathbb{Z}_2$ orbifolds. To do so, we generalize the
interpretation of Moyal star-product as a composition of symbols of
pseudodifferential operators in the case where partial derivatives
are replaced with Dunkl operators. The star-products we obtain can
be seen as globalizations of symplectic reflection algebras
(\cite{E-G}).
\end{abstract}
\maketitle
\section{Introduction}

In this paper, we construct exotic deformation quantizations of
symplectic orbifolds. Orbifolds provide a large class of examples of
topological spaces which are obtained as quotients of manifolds by
actions of compact groups. We consider a compact manifold $M$
endowed with a symplectic structure $\omega$ and with a
$\mathbb{Z}_2$ action which preserves the symplectic structure.
Given these data one can construct a $\mathbb{Z}_2$-invariant
(associative) star-product (using Fedosov method via a
$\mathbb{Z}_2$ invariant connection for instance) with the characteristic class being $\omega$. The restriction
of the invariant star-product on
$C^\infty(M)^{\integers_2}[[\hbar]]$ defines a deformation
quantization of the orbifold $X=M/\integers_2$.

\medskip

Let $(C^\infty(M)^{\integers_2}((\hbar)), \star)$ denote the star
algebra on $M/\integers_2$ with the characteristic class being
$\omega$. In \cite[Theorem 1.1]{D-E} and \cite[Theorem
VII]{PflPosTanTse}, the Hochschild cohomology of
$(C^\infty(M)^{\integers_2}((\hbar)), \star)$ was computed to be
equal to the cohomology of the corresponding inertia orbifold with
coefficient in $\complex((\hbar))$. In particular, Dolgushev and
Etingof (\cite{D-E}) conjectured\footnote{The original conjecture
states for arbitrary orbifolds. In this paper, we focus on
$\integers_2$-orbifolds.} that deformations of the algebra
$(C^\infty(M)^{\integers_2}((\hbar)), \star)$ are unobstructed.

Let $\gamma$ be the non unital element in $\integers_2$ and
$M^\gamma$ be the $\gamma$ fixed point subsets. The inertia orbifold
$\widetilde{X}$ associated to the quotient $M/\integers_2$ is equal
to $\widetilde{X}=M/\integers_2\bigsqcup M^\gamma/\integers_2$. The
Hochschild cohomology of $(C^\infty(M)^{\integers_2}((\hbar)),
\star)$ is equal to
\[
H^2(C^\infty(M)^{\integers_2}((\hbar)),
C^\infty(M)^{\integers_2}((\hbar)))=H^2(M/\integers_2)((\hbar))\bigoplus
H^0(M^\gamma_2/\integers_2)((\hbar)),
\]
where $M^\gamma_2$ is the union of components of $M^\gamma$ of
codimension 2. The Dolgushev-Etingof conjecture implies that the
algebra $(C^\infty(M)^{\integers_2}((\hbar)), \star)$ has a
deformation coming from every $\gamma$ fixed point component with
codimension 2.

\medskip

The aim of this paper is to prove that for every class in
$H^0(M^\gamma_2/\integers_2)((\hbar))$, namely every codimension 2
component of the inertia orbifold $\widetilde{X}$, we are able to
construct a deformation of the algebra
$(C^\infty(M)^{\integers_2}((\hbar)),\star)$. Moreover, there exists
a universal deformation of
$(C^\infty(M)^{\integers_2}((\hbar)),\star)$ parametrized by
$H^0(M^\gamma_2/\integers_2)((\hbar))$. This partially confirms the
conjecture of Dolgushev and Etingof in the case of $\mathbb{Z}_2$
orbifolds. Our result is not far away from the full
Dolgushev-Etingof conjecture, and the detailed relations are
explained in Remark \ref{rmk:conjecture}. The Dolgushev-Etingof
conjecture was proved by Etingof \cite{e} when an orbifold is the
cotangent bundle of a global quotient orbifold. It is the first time
that we know that a large portion of this conjecture holds true for
a large class of compact symplectic orbifolds.

\medskip

In the case where $M$ is $\mathbb{R}^{2n}$, the deformations we get
are formal versions of symplectic reflection algebras (\cite{E-G})
and our construction can be seen as a globalization of such
algebras.

\medskip

To globalize star-products on $\mathbb{R}^{2n}$, one should start
with local formulas of the star-products like the Moyal product and
Kontsevich star product \cite{Ko}. Moyal product, which deforms the
standard symplectic structure on $\mathbb{R}^{2n}$, can be described
using composition of symbols of pseudodifferential operators on
$\mathbb{R}^{n}$. One of the main ideas of the paper is to get a
generalized Moyal product formula out of composition of symbols
associated to difference-pseudodifferential operators.  Following
this approach, we replace partial derivatives with Dunkl operators
to take into account the $\mathbb{Z}_2$ action and define local
formulas for deformations of any non-commutative Poisson structures \cite{H-T}
associated with $\omega$. In this sense, we can also view our
construction as globalization of difference-pseudodifferential
operators of ``Dunkl type".

\medskip

In Section 2, we recall general material on Dunkl operators and
Dunkl pseudodifferential operators. This will allow us to construct
an operator-symbol product formula in Section 3: we will get two
families of $\mathbb{Z}_2$-local bilinear operators satisfying
properties summarized in Theorem \ref{thm:symbol-cal}. Those
operators will allow us to define a $\gamma$-local associative star
product (Proposition \ref{prop:assoc-poly}) generalizing the
standard Moyal star product. Interesting combinatorics appears in
the associativity of the new star product. The proof of this main
theorem is done in Section 5, using series expansions of
pseudodifferential calculous and explicit computations.

\medskip

Section 4 is devoted to globalization and thus to give a positive
answer to Dolgushev-Etingof conjecture. The main idea there is to
use Fedosov standard method on the complement of a tubular
neighborhood of the $\mathbb{Z}_2$ fixed point submanifold of
codimension 2. This can be done as the star product there is locally
equivalent to the Moyal product. In the neighborhood of the fixed
point submanifold of codimension 2, we use our generalized Moyal
product and Fedosov's method of quantization of fixed point
submanifolds. The fact that both the Moyal product and the
generalized Moyal product are $\gamma$-local allows us to restrict
the two deformations above on the intersections of the two open
sets, which is diffeomorphic to the tubular neighborhood of the
fixed point submanifold of codimension 2 with the fixed point
submanifold removed. We are able to glue the two deformations on the
intersection together to get a global deformation on $M/\integers_2$
as $\integers_2$ acts on the intersection freely.

\medskip

Here are some remarks and questions for future directions.
\begin{enumerate}
\item The fact that the group acting on $M$ is $\mathbb{Z}_2$ is of major
importance for our construction: if $M=\mathbb{R}^{2n}$, the
$\integers_2$ action stabilizes the two corresponding copies of
$\mathbb{R}^{n}$ and thus allows us to play with (Dunkl) operators.
Such an idea was also used by Etingof \cite{e} in his construction
of universal deformation of the cotangent bundle of a global
quotient orbifold. To extend our results to more general orbifolds,
an important question to answer is how to quantize a symplectic
orbifold when such a ``polarization" of the symplectic orbifold does
not exist.

\item One could try to generalize our results to every
$\mathbb{Z}_2$ invariant Poisson structure (and so deform the
corresponding noncommutative Poisson structure). One would expect
that with the help of the above mentioned polarization on a
$\integers_2$ orbifold, we can play with the corresponding
conjectural generalized Poisson sigma models to define the
generalized Moyal products.
\item
Another natural question is to compute Hochschild cohomology (and
$K$-theory) of our deformed algebra. It will be interesting to
develop an algebraic index theorem for our deformed algebra. We hope
to extract the information of singularities from the algebraic index
theorem.
\end{enumerate}

\medskip

\noindent{\bf Acknowledgments:} We would like to thank Calaque,
Dolgushev, and Posthuma for helpful discussions. The research of the
second author is partially supported by NSF Grant 0703775.
\section{Dunkl operator}
In this section, we briefly review the theory of Dunkl operators,
Dunkl transforms, and Dunkl pseudodifferential operators, which we
will need in this paper. We will focus ourselves to a very special
case in the theory of Dunkl operators. Most constructions and
results we are reviewing go back to Dunkl's original work \cite{du}.
We refer readers to \cite{ro} and \cite{dejeu} for the proofs of the
statements in this section.

Let $\integers_2=\{1, \gamma\}$ be the group of two elements. It acts
on the space $\reals$ by reflection. We will use
$C_c^\infty(\reals)$ to denote the space of compactly supported
smooth functions on $\reals$, and $\cals(\reals)$ to denote the
space of Schwartz functions on $\reals$. For a real parameter
$k \geq  0$, we consider the following
differential-difference operator defined by
\[
T_k(f)(x)=\frac{df}{dx}(x)+ k\frac{f(x)-f(-x)}{x}, \quad f\in
C^\infty(\reals),
\]
which is called Dunkl operator.

For the spectral  of the operator $T_k$, one considers the following
equation
\[
\left\{\begin{array}{ll}T_k(u)(x)&=-i\lambda u(x)\\
u(0)=1&\end{array}\right.
\]
for $\lambda \in \complex$.

The above equation actually has a unique solution $E_k(x,-i
\lambda)$, called Dunkl kernel given by
\[
E_k(x,-i\lambda)=j_{k-1/2}(i\lambda x)+\frac{\lambda
x}{2k+1}j_{k+1/2}(i\lambda x),
\]
where $j_\alpha$ is the ``normalized first kind Bessel function of
order $\alpha$". From the above expression, one easily see that
$E_{k}(x, -i\lambda)$ can be extended to a holomorphic function of
variable $x\in \complex,\lambda\in \complex,\operatorname{Re} k \geq
0$. One can even show that for $x,\lambda\in \reals$,
\[
|E_k(x, i\lambda)|\leq 1.
\]

We consider the following measure $\mu_k$ on $\reals$ by
\[
d\mu_k(x)=\frac{|x|^{2k}}{2^{k+1/2}\Gamma(k+1/2)}dx,
\]
with $\Gamma(x)$ the Gamma function. It is not difficult to check
that the Dunkl operator $T_k$ is skew symmetric with respect to the
$L^2$-norm associated to the measure $\mu_k$, i.e.
\[
\int_\reals T_k(f)(x)\bar{g}d\mu_k(x)=-\int_\reals
f(x)\overline{T_k(g)}(x)d\mu_k(x).
\]
For $1\leq p< \infty$, define $L^p_k(\reals)$ to be the space of
measurable complex valued functions on $\reals$ such that
\[
||f||_{p,k}=\left(\int_\reals |f(x)|^pd\mu_k(x)\right)^{1/p}<\infty.
\]

For $f\in L^1_{k}(\reals)$, define the Dunkl transform $\calf_k$ of
$f$ by
\[
\calf_k(f)(\lambda)=\int_\reals E_{k}(y,-i\lambda)f(y)d\mu_k(y).
\]
When $f$ is in $\cals(\reals)$, then
\begin{enumerate}
\item $\calf_k(f)\in C^\infty(\reals)$, and
$T_k \calf_k(f)=-\calf_k(ixf)$,
\item $\calf_k(T_k f)=i\lambda \calf_k(f)$,
\item the Dunkl transform leaves $\cals(\reals)$ invariant,
\item for all $f\in L^1_{k}(\reals)$ such that $\calf_k(f)\in
L^1_{k}(\reals)$, the inverse Dunkl transform is defined to be
\[
\calf^{-1}_k(f)(x)=\int_\reals
E_k(x,i\lambda)f(\lambda)d\mu_k(\lambda),
\]
\item for $f\in L^2_k(\reals)$,
$||\calf_k(f)||_{2,k}=||f||_{2,k}$.
\end{enumerate}
\section{Generalized pseudodifferential operators and Moyal type formula}\label{section:dunkl-weyl-algebra}
Pseudo-differential operators associated to Dunkl operators in the
case of $\integers_2$ have been studied by Dachraoui \cite{dach} and
Abdelkefi-Amri-Sifi \cite{aas}. Let $D(\reals)$ be the algebra of
differential operators on $\reals$. In this section, our goal is to
use the idea of operator-symbol calculus to define an associative
deformation of the algebra $C^\infty(\reals^2)\rtimes \integers_2$
and also $D(\reals)\rtimes \integers_2$. When one restricts such a
deformation to the
subalgebra\footnote{$\operatorname{Poly}(\reals^2)$ denotes the
algebra of polynomial functions on $\reals^2$.}
$\operatorname{Poly}(\reals^2)\rtimes \integers_2$, we obtain a
Moyal type formula for the symplectic reflection algebra introduced
by Etingof-Ginzburg \cite{E-G} in the case of $\integers_2$ action
on $\reals^2$ by reflection.
\subsection{Operator product}\label{subsec:operator}
\begin{definition}
\label{dfn:symbol}We say that a function $a(x,p)\in
C^\infty(\reals^2)$, a complex valued function on $\reals^2$,
belongs to the symbol class $\mathfrak{S}_0^m$ if for any $r,s\in
\mathbb{N}$,
\[
|\partial_p^r\partial_x^sa(x,p)|\leq C_{m,r,s}(1+|p|^2)^{(m-r)/2}.
\]
\end{definition}

\begin{definition}
\label{dfn:pseud-op} Let $a\in \mathfrak{S}_0^m$, then define
$\op_k(a)$ a linear operator on $\cals(\reals)$ by
\[
\op_k(a)(f)(x)=\int_\reals a(x,p)E_k(x,ip)\calf_k(f)(p)d\mu_k(p).
\]
\end{definition}
Dachraoui \cite[Thm. 4.1]{dach} proves the following theorem :
\begin{theorem}(\cite{dach}) Let $a\in \mathfrak{S}_0^m$, then the
operator $\op_k(a)$ associated to $a$ is a linear continuous mapping
from $\cals(\reals)$ to itself.
\end{theorem}

\begin{remark}\label{rmk:schwarz}
For $a\in \mathfrak {S}_0^0$,  \cite[Proposition 4.1]{aas} proves
that $\op_k(a)$ defines a bounded operator on $L^p_k(\reals)$ for
$1< p<\infty$.
\end{remark}

\begin{example}
For $a(x,p)=x^ip^j$, $\op_k(a)=x^iT_k^j$. We remark that though polynomials are not in the symbol class $\mathfrak{S}_0^m$, for any polynomial $a$, $\op_k(a)$ is a well defined linear operator on $\cals(\reals)$, which is sufficient for our following developments.
\end{example}

We consider the translation operator $\hat{\gamma}: f(x)\mapsto
f(-x)$. It is easily seen that $\hat{\gamma}$ is an isometry on
$L^2_k(\reals)$. We have the following observation :
\begin{lemma}\label{lem:uniqueness}
For  $ a_j, b_j\in \operatorname{Poly}(\reals^2), j=0,\dots, n$, if
$\sum_jk^j(\op_k(a_j)+\op_k(b_j)\circ \hat{\gamma})$ is the zero
operator for any $k\geq 0$, then $a_j=b_j=0, j=0,\dots, n$.
\end{lemma}
\begin{proof}
As $\op_k(\sum_j k^ja_j)+\op_k(\sum_j k^jb_j)\circ \hat{\gamma}=0$,
then
\begin{equation}\label{eq:a-b-eq}
\begin{split}
&\int_\reals
\big(\sum_jk^ja_j(x,p)\big)E_k(x,ip)\calf_k(f)(p)d\mu_k(p)\\
&\qquad +\int_\reals \big(\sum_jk^jb_j(x,p)\big)E_k(x, ip)\calf_k
(\hat{\gamma}(f))(p)d\mu_k(p)=0,
\end{split}
\end{equation}
for any $f\in \cals(\reals)$.

We notice that $\calf_k(\hat{\gamma}(f))(p)=\calf_k(f)(-p)$, then
Equation (\ref{eq:a-b-eq}) becomes
\[
\int_{\reals}
\sum_jk^j\big(a_j(x,p)E_k(x,ip)+b_j(x,-p)E_{k}(x,-ip)\big)
\calf_k(f)(p)d\mu_k(p)=0,
\]
for any $f\in \cals(\reals)$. Therefore, we conclude that
\[
\sum_jk^j(a_j(x,p)E_k(x,ip)+b_j(x,-p)E_{k}(x,-ip))=0,
\]
for any $x,p\in \reals$. If we consider the above equation at $k=0$,
then
\[
a_0(x,p)\exp(ixp)+b_0(x,-p)\exp(-ixp)=0.
\]
From the above equation, we have that
\[
\partial_xa(x,p)
b(x,-p)-a(x,p)\partial_xb(x,-p)=-2ip a(x,p)b(x,-p).
\]
By comparing
the leading terms on both sides, we can quickly
conclude that $a_0=b_0=0$. By induction, we conclude that
$a_j=b_j=0$ for $j=0,\dots,n$.
\end{proof}

To motivate the main result of this section, we introduce the
following notion of a $\gamma$-local operator. (Recall that $\gamma$
acts on $\reals$ by reflection.)
\begin{definition}\label{dfn:e-operator}A linear operator $D$ on
$C^\infty(\reals^2)$ is called $\gamma$-local if for any $f\in
C^\infty(\reals^2)$, $D(f)(x,p)$ is determined completed by finitely
many jets of $f$ at $(x,p)$, $(-x,p)$, $(x,-p)$, and $(-x,-p)$. In
general, a $k$-linear operator $D$ on $C^\infty(\reals^2)$ is called
$\gamma$-local, if for any $f_1,\dots, f_k \in C^\infty(\reals^2)$,
$D(f_1, \dots, f_k)(x,p)$ is determined by finitely many jets of
$f_1, \dots, f_k$ at $(x,p)$, $(-x,p)$, $(x,-p)$, and $(-x,-p)$.
\end{definition}

\begin{example}Let us list some examples of $\gamma$-local operators.
\label {ex:e-local}
\begin{enumerate}
\item Differential operators on $\reals$ are
$\gamma$-local.
\item The partial translation operator $\sigma_i:C^\infty(\reals^2)\to C^\infty(\reals^2)$ for $i=1,2$ with
$\sigma_1(f)(x,p):=f(-x,p)$ and $\sigma_2(f)(x,p)=f(x-p)$ are
$\gamma$-local.
\item The difference operators $\tilde{\partial}_x, \tilde{\partial}_p:C^\infty(\reals^2)\to
C^\infty(\reals^2)$ with
$\tilde{\partial}_x(f)(x,p)=(f(x,p)-f(-x,p))/x$ and
$\tilde{\partial}_p(f)(x,p)=(f(x,p)-f(x,-p))/p$ are $\gamma$-local.
We observe that $\partial_x+\tilde{\partial}_x$
(and $\partial_p+\tilde{\partial}_p$)is the Dunkl operator $T_1$ acting
on the $x$-variable (and the $p$-variable), and is also $\Gamma$-local.
\end{enumerate}
\end{example}

\begin{proposition}
\label{prop:e-local-alg}The space of $\gamma$-local operators on
$C^\infty(\reals^2)$ is an associative algebra under composition.
\end{proposition}
\begin{proof}
This is a straightforward check.
\end{proof}

The main result of this section can be summarized into the following
Theorem.
\begin{theorem}\label{thm:symbol-cal}There are 2 families of
$\gamma$-local bilinear operators $C^1_{j,l}$ and $C^2_{j,l}$ on
$C^\infty(\reals^2)$ satisfying
\begin{enumerate}
\item For two polynomials $a_1$ and $a_2$ of degrees $(m_1,n_1)$ and $(m_2,n_2)$,
$C^0_{j,l}(a_1,a_2)$ and $C^1_{j,l}(a_1,a_2)$ are again polynomials
of degree $(m_1+m_2-j, n_1+n_2-j)$.
\item $C^0_{j,l}$ and $C^1_{j,l}$ vanish when $l>j$.
\item For two polynomials $a_1(x,p)$ and $a_2(x,p)$,
\[
\op_k(a_1)\circ \op_k(a_2)=\sum_{j,l}k^l\Big(
\op_k\big(C^0_{j,l}(a_1,a_2)\big)+\op_k\big(C^1_{j,l}(a_1,a_2)\big)\circ
\hat{\gamma}\Big).
\]
We observe that for any given $a_1,a_2$, the above sum is actually
finite and therefore well defined.
\end{enumerate}
\end{theorem}

The proof of this theorem will be given in Section
\ref{sec:proof-thm}. In the left of this section, we will provide an
explicit formula for each bilinear operator $C^i_{jl}$. In
particular, when $l=0$, $C^1_{j,0}$ vanishes and $C^0_{j,0}$ is the
$j$-th component of the Moyal product,
\begin{equation}\label{eq:moyal}
C^1_{j,0}(a_1,
a_2)=\frac{(-i)^j}{j!}\partial_p^j(a_1)\partial_x^j(a_2).
\end{equation}

From this, we can see that the above operator-symbol calculus
defines a deformation of the crossed production of $D(\reals)\rtimes
\integers_2$.
\subsection{A coproduct structure on $\operatorname{Poly}(\reals)$}

We consider a coproduct structure on the algebra of polynomials of
one variable, which is useful in describing the operators
$C^i_{jl}$.

Define $\Delta$ to be a linear map from $\operatorname{Poly}(\reals)$ to
$\operatorname{Poly}(\reals)\otimes_\complex \operatorname{Poly}(\reals)$ by
\[
\Delta(f)(x,y):=\frac{f(x)-f(y)}{x-y}.
\]

Observe that $f(x)-f(y)$ is divisible by $x-y$, and therefore
$\Delta$ is well defined.

The following is a list of properties of the operator $\Delta$,
which can be checked routinely.

\begin{proposition}
\label{prop:delta}The operator
$\Delta:\operatorname{Poly}(\reals)\to
\operatorname{Poly}(\reals)\otimes \operatorname{Poly}(\reals)$
satisfies the following properties.
\begin{enumerate}
\item coassociative, i.e.
\[
(\Delta\otimes 1)\Delta=(1\otimes \Delta)\Delta: \operatorname{Poly}(\reals)\to
\operatorname{Poly}(\reals)\otimes \operatorname{Poly}(\reals)\otimes \operatorname{Poly}(\reals);
\]
\item Leibnitz rule, i.e.
\[
\Delta(fg)=(f\otimes 1)\Delta(g)+\Delta(f)(1\otimes g);
\]
\item $\Delta(f)(x,x)=f'(x)=D(f)(x)$, and
$\Delta(f)(x,-x)=(f(x)-f(-x))/2x=1/2\tilde{D}(f)(x)$, and
$T_k(f)(x)=(D+k\tilde{D})(f)(x)$;
\item $\Delta(f)$ is a symmetric function of $2$ variables;
\item $\Delta$ extends to be a linear map $\Delta:C^\infty(\reals)\to
C^\infty(\reals)\hat{\otimes }C^\infty(\reals)$ satisfying the same
properties (1)-(4), where $\hat{\otimes}$ is the complete
topological tensor product.
\end{enumerate}
\end{proposition}

\begin{remark}
According to Proposition \ref{prop:delta}, (2), the operator
$\Delta$ is a Hochschild cocycle of $\operatorname{Poly}(\reals)$
with coefficient in $\operatorname{Poly}(\reals)\otimes
\operatorname{Poly}(\reals)$. By the Koszul complex, we can compute
that the Hochschild cohomology $H^1(\operatorname{Poly}(\reals),
\operatorname{Poly}(\reals)^{\otimes 2})$ is equal to
$\operatorname{Poly}(\reals)$. Under this identification, $\Delta$
is mapped to the unit of $\operatorname{Poly}(\reals)$.
\end{remark}

\begin{remark}
For $\reals^n$, we can generalize $\Delta$ to a cocycle $\Delta_n:
\operatorname{Poly}(\reals^n)^{\otimes n}\to \operatorname{Poly}(\reals^n)^{\otimes 2}$ by
\begin{eqnarray*}
&\Delta_n(f_1, \dots, f_n)(x,y)\\
:=&\frac{(f_1(x_1, \dots, x_n)-f_1(y_1, x_2, \dots, x_n))(f_2(y_1,
x_2, \dots, x_n)-f_2(y_1, y_2, x_3, \cdots, x_n))\cdots (f_n(y_1,
\dots, y_{n-1},x_n)-f_n(y_1, \dots, y_n))}{(x_1-y_1)\cdots
(x_n-y_n)},
\end{eqnarray*}
where $x=(x_1, \dots, x_n)$ and $y=(y_1, \dots, y_n)$.
\end{remark}
\subsection{Formulas for asymptotic
expansion}\label{subsec:bilinear-operators}
We will give explicit expressions for $C^i_{j,l}$, $i=1,2$, which involves interesting combinatorics.

We start with considering the linear equation
\begin{equation}\label{eq:linear-partition}
y_0+y_1+\cdots + y_{l}=j-l.
\end{equation}
Let $P_{j-l, l}$ be the set of integer solutions to Equation
(\ref{eq:linear-partition}) where $y_0, y_l$ are nonnegative and
$y_1, \dots, y_{l-1}$ are positive.

Let $D(f)(x)=f'(x)$ and $\tilde{D}(f)(x)=(f(x)-f(-x))/x$.

For an element $\nu\in P_{m,n}$, define $B_\nu$ a linear operator on
$\operatorname{Poly}(\reals)$ by
\[
B_\nu(f)(x):=D^{y_n}\circ \tilde{D} \circ D^{y_{n-1}}\cdots
D^{y_1}\circ \tilde{D}\circ D^{y_0}(f)(x).
\]

For $n\in \mathbb{N}$, define $n_0$ to be the number of positive even numbers
less than or equal to $n$, and $n_1$ to be the number of positive odd numbers
less than or equal to $n$. Obviously, $n=n_0+n_1$. Given $\nu\in
P_{m,n}$, we define $\Lambda_0=y_0+\sum_{\text{even}\ i }y_i$, and
$\Lambda_1=\sum_{ \text{odd}\ i}y_i$. We have
$\Lambda_0+\Lambda_1=m$. Define $A_\nu$ a linear operator on
$\operatorname{Poly}(\reals)$ by
\[
\Delta^{m+n}(f)(\underbrace{x, \dots, x}_{\Lambda_0+n_0+1},
\underbrace{-x, \dots, -x}_{\Lambda_1+n_1}).
\]
By the associativity of $\Delta$ (Prop.
\ref{prop:delta}, (1)), define $\Delta^k(f)=(\Delta\otimes 1\otimes
\cdots \otimes 1)\cdots (\Delta\otimes 1)\Delta(f)$. And according to Prop. \ref{prop:delta} (4), $\Delta^{k}(f)$ is a symmetric function of $k+1$ variables.

In order to define $C^i_{j,l}$, which are bilinear operators on
$\operatorname{Poly}(\reals^2)$, we lift $A_\nu$ and $B_{\nu}$ on $\operatorname{Poly}(\reals^2)$
by applying $A_{\nu}$ on the variable $p$ and $B_{\nu}$ on the
variable $x$.

Now we are ready to define $C^i_{j,l}$.
\begin{enumerate}
\item[I.] $C^0_{j,l}$. The bilinear operator $C^0_{j,l}$ vanishes if $l$ is odd, and when
$l$ is even,
\[
C^0_{j,l}(a_1,a_2):=(-i)^j\sum_{\nu \in P_{j-l,l}}A_{\nu}(a_1)(x,p)B_\nu(a_2)(x,p).
\]

\item[II.] $C^1_{j,l}$. The bilinear operator $C^1_{j,l}$ vanishes if $l$ is even, and when
$l$ is odd,
\[
C^1_{j,l}(a_1, a_2)(x,p):=(-i)^{j}\sum_{\nu\in P_{j-l,l}} A_\nu(a_1)(x,p)B_{\nu}(a_2)(x,-p).
\]
\end{enumerate}

We point out that with the expression of $C^i_{ij}$, Theorem
\ref{thm:symbol-cal}, (1) follows obviously by the definition of
$A_\nu$ and $B_\nu$. Furthermore, one notices that if $j-l<l-1$, then
$P_{j-l,l}$ is an empty set, and therefore $C^i_{j,l}$ vanishes.
This gives a stronger version of Theorem \ref{thm:symbol-cal}, (2).

From the above discussion, we are left to prove part (3) of Theorem
\ref{thm:symbol-cal}. This is an interesting application of
operator-symbol calculus  and the detail will be in Section
\ref{sec:proof-thm}. In particular, we will explain how we obtain
the operators $A_\nu$ and $B_\nu$.
\subsection{A ``Moyal" formula}
Motivated by the result of Theorem \ref{thm:symbol-cal}, we
introduce the following algebra.
\begin{definition}
\label{dfn:deformation}Define the following product $\star$ on
$C^\infty(\reals^2)\rtimes_{\complex} \integers_2[[\hbar_1,
\hbar_2]]$ by
\begin{enumerate}
\item $\star$ is $\complex[[\hbar_1, \hbar_2]]$ linear;
\item For $a_1, a_2\in C^\infty(\reals^2)$, $a_1\star a_2$ is
defined by
\[
a_1\star
a_2=\sum_{j,l}\hbar_1^j\hbar_2^l(C^0_{j,l}(a_1,a_2)+C^1_{j,l}(a_1,
a_2)\gamma).
\]
\end{enumerate}
\end{definition}

As we have explained at the end of Section \ref{subsec:operator},
when $\hbar_2=0$, the above product $\star$ reduces back the
standard Moyal product. Hence, we can view
$(C^\infty(\reals^2)\rtimes \integers_2[[\hbar_1, \hbar_2]], \star)$
as a deformation of the crossed product of the Weyl algebra
$\mathbb{W}_2$ with $\integers_2$. Furthermore, we point out that as
$C^i_{j,l}=0$ when $l>j$, we can allow $\hbar_2$ be a complex number
in $\complex$ rather than a formal parameter. In this way, we can
also view $(C^\infty(\reals^2)\rtimes \integers_2[[\hbar_1,
\hbar_2]], \star)$ as a formal deformation quantization of the
crossed product algebra $C^\infty(\reals^2)\rtimes \integers_2$
along the noncommutative Poisson structure $\pi+\hbar_2 \pi\gamma$
on $C^\infty(\reals^2)\rtimes \integers_2$ as we introduced in
\cite{H-T}.

\begin{lemma}\label{lem:poly-approx}For any $(x_0,p_0)\in \reals^2$, and $f\in
C^\infty(\reals^2)$, given $m,n\in \mathbb{N}\cup \{0\}$, there is a
polynomial $g_{m,n}\in \operatorname{Poly}(\reals^2)$ such that
$\partial_x^i\partial_p^jf$ agrees with
$\partial_x^i\partial_p^jg_{m,n}$ at $(x_0,p_0)$, $(x_0,-p_0)$,
$(-x_0,p_0)$, and $(-x_0,-p_0)$ for $0\leq i\leq m, 0\leq j\leq n$.
\end{lemma}
\begin{proof}  We divide our proofs into 4 different
situations according to $(x_0, p_0)$.
\begin{enumerate}
\item $x_0=p_0=0$,
\item $x_0\ne 0$ and $p_0=0$,
\item $x_0=0$ and $p_0\ne 0$,
\item $x_0\ne 0$ and $p_0\ne 0$.
\end{enumerate}

\noindent{\bf Case (1)}. For any $m,n\in \naturals$, define
\[
g_{m,n}=\sum_{0\leq i\leq m, 0\leq j\leq
n}\frac{1}{i!j!}\partial_x^i\partial^j_p(f)(0,0)x^ip^j.
\]
It is easy to check $\partial_x^i\partial^j_p g_{m,n}$ agrees
$\partial_x^i\partial_p^jf$ at $(0,0)$ for $0\leq i\leq m, 0\leq
j\leq n$.

\noindent{\bf Case (2) and (3)}. The proof for these two cases are
exactly same. Therefore, we will only prove Case (2). Define
\[
g_1=\sum_{0\leq i\leq m, 0\leq j\leq
n}\frac{1}{i!j!}\partial_x^i\partial^j_p(f)(x_0,0)(x-x_0)^ip^j.
\]
Define $g_{m,n}=g_1+(x-x_0)^{m+1}g_2$ where $g_2$ is some polynomial
to be determined. It is easy to check that
$\partial_x^i\partial^j_pg_{m,n}(x_0,0)$ agrees with
$\partial^i_x\partial^j_pf(x_0,0)$. We proceed to look for $g_2$
such that $\partial_x^i\partial_p^jg_{m,n}(-x_0,0)$ agrees with
$\partial_x^i\partial_p^jf(-x_0,0)$. We write
\[
g_2=\sum_{1\leq s\leq m,1\leq t\leq n}1/{s!t!}a_{st}(x+x_0)^sp^t.
\]
We need to solve $a_{st}$. From the requirement that
$\partial_x^i\partial_p^jg_{m,n}(-x_0,0)=\partial_x^i\partial_p^jf(-x_0,0)$,
we know that
\begin{equation}\label{eq:x-0}
\partial_x^i\partial_p^j(g_1)(-x_0,0)+\left(\begin{array}{c}i\\
k\end{array}\right)\partial_x^{i-k}(x-x_0)^{m+1}
\partial_x^{k}\partial_p^jg_2(-x_0,0)=\partial_x^i\partial_p^jf(-x_0,0).
\end{equation}
If we order $a_{st}$ lexicographically, then it is not difficult to
see that the above equations for $1\leq i\leq m, 1\leq j\leq n$
define a system of linear equations for variable $a_{st}$. We notice
that in Eq. (\ref{eq:x-0}), the leading term is $a_{ij}$ with
coefficient $(-2x_0)^{m+1}$. When $i$ and $j$ vary, we have a system
of linear equations whose coefficient matrix is an upper triangular
matrix with a nonzero number $(-2x_0)^{m+1}$ at every entry of the
diagonal. This implies that we have a unique solution for $a_{st}$,
and therefore a solution for $g_{m,n}$.

\noindent{\bf Case (4)}. Following the proof of Case (2),
we construct $g$ step by step. Firstly, define $g_0$ to be
\[
g_0=\sum_{0\leq i\leq m, 0\leq j\leq
n}\frac{1}{i!j!}\partial_x^i\partial_p^jf(x_0,
y_0)(x-x_0)^i(p-p_0)^j.
\]
We now look for $g_1$ of the form $\sum_{0\leq i\leq m, 0\leq j\leq
n}1/i!j! a_{ij}(x+x_0)^i(p-p_0)^j$ such that
$\partial_x^i\partial_p^j(g_0+(x-x_0)^{m+1}g_1)$ agrees with
$\partial_x^i\partial_p^jf$ at both $(x_0,p_0)$ and $(-x_0,p_0)$ for
$0\leq i\leq m, 0\leq j\leq n$. We notice that it is always true
that
$\partial_x^i\partial_p^j(g_0+(x-x_0)^{m+1}g_1)(x_0,p_0)=\partial_x^i\partial_p^jf(x_0,p_0)$
for $0\leq i\leq m, 0\leq j\leq n$. By the same arguments as in the
proof of Case (2), we can find a unique family $a_{ij}$ such that
$\partial_x^i\partial_p^j(g_0+(x-x_0)^{m+1}g_1)(-x_0,p_0)$ is same
to $\partial_x^i\partial_p^j(f)(-x_0,p_0)$ for $0\leq i\leq m, 0\leq
j\leq n$.

We next look for $g_2$ of the form $\sum_{0\leq i\leq m, 0\leq j\leq
n}1/i!j! b_{ij}(x-x_0)^i(p+p_0)^j$ such that
$\partial_x^i\partial_p^j(g_0+(x-x_0)^{m+1}g_1+(p-p_0)^{n+1}g_2)$
agrees with $\partial_x^i\partial_p^jf$ at $(x_0, p_0)$, $(-x_0,
p_0)$, $(x_0, -p_0)$ for $0\leq i\leq m, 0\leq j\leq n$. Again, it
not difficult to check that the partial derivatives of these two
functions agree at $(x_0, p_0)$ and $(-x_0, p_0)$ no matter what
$g_2$ is like. With the above arguments, we know that there exists a
unique solution for $b_{st}$ such that the derivatives of the two
functions agree at $(x_0,-p_0)$.

Continuing the above procedure, we look for $g_3$ of the form
$\sum_{0\leq i\leq m, 0\leq j\leq n}1/i!j!
c_{ij}(x+x_0)^i(p+p_0)^j)$ such that
$\partial_x^i\partial_p^j(g_0+(x-x_0)^{m+1}g_1+(p-p_0)^{n+1}g_2+(x-x_0)^{m+1}(p-p_0)^{n+1}g_3)$
agrees with $\partial_x^i\partial_p^jf$ at $(x_0, p_0), (-x_0, p_0),
(x_0, p_0), (-x_0, -p_0)$ for $0\leq i\leq m, 0\leq j\leq n$. Again
the two functions have the same derivatives at $(x_0, p_0),
(-x_0,p_0), (x_0, -p_0)$ no matter what $g_3$ is like. The same
arguments as in the proof of Case (2) shows that there is a unique
solution for $c_{ij}$.

In summary, we have fund a function
$g_{m,n}=g_0+(x-x_0)^{m+1}g_1+(p-p_0)^{n+1}g_2+(x-x_0)^{m+1}(p-p_0)^{n+1}g_3$
such that $\partial_x^i\partial_p^jf$ agrees with
$\partial_x^i\partial_p^jg_{m,n}$ at $(x_0, p_0)$, $(-x_0, p_0)$,
$(x_0,-p_0)$, and $(-x_0, -p_0)$ for $0\leq i\leq m, 0\leq j\leq n$.
\end{proof}
\begin{proposition}
\label{prop:assoc-poly}The product $\star$ is associative on
$C^\infty(\reals^2)\rtimes \integers_2[[\hbar_1, \hbar_2]]$. For
$e^{i\theta}\in U(1)$, the map $x\mapsto e^{i\theta}x, p\mapsto
e^{-i\theta}p$ defines a $U(1)$ action on the algebra
$(C^\infty(\reals^2)\rtimes \integers_2[[\hbar_1, \hbar_2]],\star)$
\end{proposition}
\begin{proof}
We observe that $\operatorname{Poly}(\reals^2)\rtimes \integers_z$ is closed under
$\star$. If $a_i$ ($i=1,2,3$) are monomials of degrees $(m_i,n_i)$,
then $\sum_{l}k^l(C^0_{j,l}(a_1, a_2)+C^1_{j,l}(a_1,a_2)\gamma)$ is
the degree $(m_1+m_2-j,n_1+n_2-j)$ in the expansion of
$\op_k(a_1)\circ\op_k(a_2)$. Therefore,
\begin{eqnarray*}
&\sum_{j_1+j_2=j}\sum_{l}k^{l}&\sum_{l_1+l_2=l}\Big(\op_k \big(C^0_{j_1,l_1}(C^0_{j_2,l_2}(a_1,a_2),a_3)
+C^1_{j_1,l_1}(C^1_{j_2,l_2}(a_1,a_2),\gamma(a_3))\big)\\
&&+\op_k \big(C^0_{j_1,l_1}(C^1_{j_2,l_2}(a_1,a_2),
\gamma(a_3))+C^1_{j_1,l_1}(C^0_{j_2,l_2}(a_1,a_2),a_3)\big)\gamma\Big)
\end{eqnarray*}
is the degree $(m_1+m_2+m_3-j,n_1+n_2+n_3-j)$ component of the
expansion of $(\op_k(a_1)\circ\op_k(a_2))\circ \op_k(a_3)$.

As the composition between operators on $\cals(\reals)$ is
associative, by Theorem \ref{thm:symbol-cal} and Lemma
\ref{lem:uniqueness}, we conclude that the product $\star$ on
$\operatorname{Poly}(\reals^2)\rtimes \integers_2$ is associative by comparing
components with degree $(m_1+m_2+m_3-j,n_1+n_2+n_3-j)$ and power
$k^l$ in the expansions of $(\op_k(a_1)\circ \op_k(a_2))\circ
\op_k(a_3)$ and $\op_k(a_1)\circ(\op_k(a_2)\circ \op_k(a_3))$.

To prove that $\star$ is associative on $C^\infty(\reals^2)\rtimes
\integers_2$, it is sufficient to check that $\star$ is associative
at every point $(x,p)$ up to any $\hbar_1^j\hbar_2^l$. We notice
that $C^i_{j,l}(a_1,a_2)(x,p)$ is determined by the values of
$a_1,\partial_pa_1, \dots,
\partial_p^ja_{1}$ at $(x,p)$ and $(x,-p)$, together with values of
$a_2, \partial_xa_2, \dots,
\partial_x^j a_{2}$ at $(x,p), (-x,p), (x,-p)$, and $(-x,-p)$.
Therefore to check $((a_1\star a_2)\star a_3)(x,p)$ agrees with
$(a_1\star (a_2\star a_3))(x,p)$ up to degree $\hbar_1^j\hbar_2^l$,
it sufficient to check $(b_1\star b_2)\star b_3(x,p)$ agrees with
$b_1\star (b_2\star b_3)(x,p)$ up to degree $\hbar_1^j\hbar_2^l$ for
polynomials $b_1,b_2,b_3$ where the values of
$\partial_x^s\partial_p^tb_i$ at $(x,p), (-x,p), (x,-p), (-x,-p)$
agree with the corresponding values of $\partial_x^s\partial_p^ta_i$
for $i=1,2,3, 1\leq s,t\leq j$. Hence by the associativity of
$\star$ on $\operatorname{Poly}(\reals^2)\rtimes \integers_2$ and
Lemma \ref{lem:poly-approx}, we conclude that $\star$ is associative
on $C^\infty(\reals^2)\rtimes\integers_2$.

For the action of $t=e^{i\theta}$, we notice that $t=e^{i\theta}$
acts on operators $A_\nu$ and $B_\nu$ with eigenvalues $t^{-j}$ and
$t^{j}$. Therefore, one can quickly check that $C^i_{jl}$ is a $U(1)$
invariant bilinear operator on $C^\infty(\reals^2)\rtimes
\integers_2[[\hbar_1, \hbar_2]]$ for any $i,j,l$. Therefore, $U(1)$
acts on $C^\infty(\reals^2)\rtimes \integers_2[[\hbar_1, \hbar_2]]$
by algebra automorphisms.
\end{proof}

\begin{remark}
The algebra $(C^\infty(\reals^2)\rtimes \integers_2[[\hbar_1,
\hbar_2]], \star)$ is the formal version of a symplectic
reflection algebra \cite{E-G} for $\integers_2$ action on the standard
symplectic vector space $\reals^2$. Theorem \ref{thm:symbol-cal}
gives an operator interpretation of this symplectic reflection
algebra and furthermore a Moyal type expansion formula.
\end{remark}

Let $P=(1+\gamma)/2\in C^\infty(\reals^2)\rtimes \integers_2$.
Consider the subspace of $A=(C^\infty(\reals^2)\rtimes
\integers_2[[\hbar_1, \hbar_2]],\star)$ defined by $P\star A\star
P=P\star C^\infty(\reals^2)\rtimes \integers_2[[\hbar_1, \hbar_2]]
\star P$. In \cite{E-G}, Etingof and Ginzburg proved that $P\star
A\star P$ is Morita equivalent to $A$. In particular, one can
quickly check that the space $P\star A\star P$ as a vector space is
isomorphic to $C^\infty(\reals^2)^{\integers_2}[[\hbar_1,
\hbar_2]]$. Via the natural identification,
\[
a\in C^\infty(\reals^2)^{\integers_2}[[\hbar_1, \hbar_2]]\mapsto
aP\in C^\infty(\reals^2)\rtimes \integers_2[[\hbar_1, \hbar_2]],
\]
$C^\infty(\reals^2)^{\integers_2}[[\hbar_1, \hbar_2]]$ is equipped
with a star-product which we will again denote by $\star$. We call
this algebra Dunkl-Weyl algebra $\mathbb{D}_2$, which is called the
spherical subalgebra by Etingof and Ginzburg \cite{E-G}. By
Proposition \ref{prop:assoc-poly}, we conclude that the Dunkl-Weyl
algebra $\mathbb{D}_2$ is an associative algebra with a natural
$U(1)$ action. \\

\noindent{\bf Notation:} We use $\complex((\hbar_1)) ((\hbar_2))$ to denote the space of all series of the form
\[
\sum_{t\leq j}a_j(\hbar_1)\hbar_2^j
\]
for some $t\in \integers$, and $a_j\in \complex((\hbar_1))$.
In the later applications, we many times will work
with the algebra $\mathbb{D}_2\otimes_{\complex[[\hbar_1,
\hbar_2]]}\complex((\hbar_1))((\hbar_2))$, which will be denoted by
$\mathbb{D}_2((\hbar_1))((\hbar_2))$.
\section{Quantization of $\mathbb{Z}_2$-orbifold}
In this section, we consider deformation quantization of
$\integers_2$-orbifolds. Let $M$ be a symplectic manifold with a
symplectic $\integers_2$ action. As $\integers_2$ is finite, we can
always find a $\integers_2$ invariant symplectic connection on $M$.
Using Fedosov's method, we can construct a $\integers_2$ invariant
star-product $\star$ on $C^\infty(M)[[\hbar]]$. (As our construction
is local, it works more generally for an orbifold which locally is a
quotient of a $\integers_2$ action.) The restriction of the
invariant star-product on $C^\infty(M)^{\integers_2}[[\hbar]]$
defines a deformation quantization of the orbifold
$X=M/\integers_2$. We use $A^{((\hbar))}_{M/\integers_2}$ to denote
the quantized algebra on $M/\integers_2$ with the characteristic
class equal to $\omega$ (with $A^{((\hbar))}_{M/\integers_2}$, we
refer to the algebra
$C^\infty(M)^{\integers_2}\otimes_{\complex}\complex((\hbar))$ with
the extended star-product $\star$).

According to \cite[Theorem 1.1]{D-E} and \cite[Theorem
VII]{PflPosTanTse}, the Hochschild cohomology of
$A^{((\hbar))}_{M/\integers_2}$ is equal to the cohomology of the
corresponding inertia orbifold with coefficient in
$\complex((\hbar))$. In the case of $M/\integers_2$, the
corresponding inertia orbifold is defined to be
$\tilde{X}:=M/\integers_2\sqcup M^\gamma/\integers_2$, where
$M^\gamma$ is the fixed submanifold of the group element $\gamma$ in
$\integers_2$. If $M^\gamma$ has several components maybe of
different dimensions, we will take the disjoint union of all
components. We use $\ell$ to denote the codimension of $M^\gamma$ in
$M$, and $\ell$ is a locally constant function on $X$. We point out
that $\integers_2$ acts on $M^\gamma$ trivially, but we will view
$M^\gamma$ as an orbifold with a global stabilizer group
$\integers_2$. We have
\begin{equation}\label{eq:hoch-coh}
H^\bullet(A^{((\hbar))}_{M/\integers_2},
A^{((\hbar))}_{M/\integers_2})=H^{\bullet-\ell}(\tilde{X},
\complex((\hbar))).
\end{equation}

Looking at Equation (\ref{eq:hoch-coh}), we conclude that the second
Hochschild cohomology of $A^{((\hbar))}_{M/\integers_2}$ is equal to
a direct sum of $H^2(M/\integers_2, \complex((\hbar)))$ and
$H^0(M^\gamma_2/\integers_2, \complex((\hbar)))$ for the components
$M^\gamma_2$ of $M^\gamma$ with codimension 2 (we have degree 0
cohomology on $M_2^\gamma$ because of the degree shifting in
Equation (\ref{eq:hoch-coh})). From the experience of deformation
quantization of a symplectic manifold, we know that the component
$H^2(M/\integers_2, \complex((\hbar)))$ of
$H^2(A^{((\hbar))}_{M/\integers_2}, A^{((\hbar))}_{M/\integers_2})$
corresponds to isomorphism classes of $\integers_2$ invariant
deformation quantizations on $M$. In the following of this section,
we construct deformations of $A^{((\hbar))}_{M/\integers_2}$
corresponding to $H^0(M^\gamma_2/\integers_2, \complex((\hbar)))$.
This gives a partial positive answer to \cite[Conjecture 1]{D-E} in
the case of $\integers_2$ orbifolds. We construct a deformation of
$A^{((\hbar))}_{M/\integers_2}$ in 3 steps,
\begin{enumerate}
\item Dunkl-Weyl algebra bundle,
\item Quantization of punctured disk bundle,
\item Global quantization.
\end{enumerate}

We briefly explain the strategy before we go into the details of the
construction.  In the first step, we will quantize the normal bundle
of the $\gamma$ fixed point submanifold with codimension 2.
Quantization of normal bundle of a fixed point submanifold has been
considered by Fedosov \cite{fe:g-index} and Kravchenko \cite{Kr}.
Here the new input is that along the fiber direction of the normal
bundle, we will use the Dunkl-Weyl algebra introduced at the end of
Section \ref{section:dunkl-weyl-algebra}.  The main result will be
that with the new algebra $\mathbb{D}_2((\hbar_1))((\hbar_2))$, the
construction of Fedosov \cite{fe:g-index} and Kravchenko \cite{Kr}
has a natural generalization and we obtain a flat connection on the
associated Dunkl-Weyl algebra bundle. This first step can be viewed
as a quantization of a tubular neighborhood of the $\gamma$ fixed
point submanifold with codimension 2. In order to extend this
quantization of a tubular neighborhood of the $\gamma$ fixed point
submanifolds, in Step 2, we restrict the quantization we obtained in
Step 1 to a punctured tubular neighborhood of the $\gamma$ fixed
point submanifold with the zero section removed. We are allowed to
restrict this quantization because of the locality of the product
$\star$ on $\mathbb{D}_2((\hbar_1))((\hbar_2))$ discussed in Section
\ref{section:dunkl-weyl-algebra}, Theorem \ref{thm:symbol-cal}. An
important property of the punctured tubular neighborhood is that the
$\integers_2$ action on it is free, and there is no fixed point.
Therefore, quantizations of such a punctured neighborhood can be
classified by Fedosov's theory without any extra contribution from
the fixed point submanifold. In Step 3, we will extend the
quantization obtained in Step 1 of the tubular neighborhood of the
$\gamma$ fixed point submanifold with codimension 2 to the whole
orbifold. Here the key is that with the study in Step 2, we can
regularize the quantization obtained in Step 1 on the punctured
tubular neighborhood. Namely, it is isomorphic to some standard
quantization of the punctured tubular neighborhood using Fedosov's
construction via the characteristic classes developed by Fedosov
\cite{fe:book} and Kravchenko \cite{Kr}. We point out the above
strategy is possible to be generalized by replacing the Dunkl-Weyl
algebra $\mathbb{D}_2((\hbar_1))((\hbar_2))$ by the spherical
subalgebra of other symplectic reflection algebras \cite{E-G} if we
know the product is ``local".

\subsection{Dunkl-Weyl algebra bundle}\label{sec:dunkl-weyl-alg}
We consider the collection of connected components of $M^\gamma$
which are of codimension 2, and we denote it by $M^\gamma_2$. The
symplectic orthogonal space of $TM^\gamma_2$ in $TM|_{M^\gamma_2}$
defines a normal bundle $N$ of $M^\gamma_2$ in $M$. $N$ inherits a
$\integers_2$ action from the $\integers_2$ action on $M$. The
restriction of the symplectic form $\omega$ to $N$ makes $N$ a
$\integers_2$ equivariant symplectic vector bundle with the
symplectic structure $\omega^N$. We will fix a global $\integers_2$
invariant compatible almost complex structure on $M$. (Such an
almost complex structure always exists.) An invariant almost complex
structure makes $N$ into a $\integers_2$ equivariant hermitian line
bundle. In particular, the corresponding principal bundle $P$
associated to $N$ is a principal $U(1)$ bundle. By Proposition
\ref{prop:assoc-poly}, $U(1)$ naturally acts on the Dunkl-Weyl
algebra $(\mathbb{D}_2((\hbar_1))((\hbar_2)),\star)$. Therefore, we
define the following Dunkl-Weyl algebra bundle over $M^\gamma_2$ by
\[
\calv:=P\times_{U(1)}\mathbb{D}_2((\hbar_1))((\hbar_2)).
\]

We have constructed a bundle $\calv$ of infinitely dimensional
algebras over a symplectic manifold $M^\gamma_2$. The hermitian
connection on the principal bundle $P$ induces a connection on the
Dunkl-Weyl algebra bundle. We exhibit this connection in local
coordinates. Let $x^\nu\ (\nu=1, \dots, 2n-2)$ be coordinates on
$M^\gamma_2$ and $z,\bar{z}$ be coordinates along the fiber
direction. The hermitian connection $\nabla^N$ on $N$ can be written
as
\[
\nabla^N_{\frac{\partial}{\partial
x^\nu}}\partial_z=i\Gamma_{\nu}(x)\partial_z,\qquad
\nabla^N_{\frac{\partial}{\partial
x^\nu}}\partial_{\bar{z}}=-i\Gamma_\nu(x) \partial_{\bar{z}},
\]
where $\Gamma_\nu$ is a real valued function on $M^\gamma_2$.

The induced connection $\partial^N$ on $\calv$ is defined by
\[
\partial^N \xi=dx^i \otimes \big(\frac{\partial \xi}{\partial
x^i}+\frac{i}{2\hbar_1}[\Gamma_i z\bar{z}, \xi ]_\star \big),\qquad
\xi \in \Gamma(\calv),
\]
where $[\ ,\ ]_\star$ is the star-commutator.

Let $R^N_{\nu_1\nu_2}$ be the curvature tensor associated to the
hermitian connection $\nabla$. Then one can quickly compute that
\[
\partial^N\circ \partial^N(\xi)=\frac{1}{2\hbar_1}[dx^{\nu_1}\wedge
dx^{\nu_2}R^N_{\nu_1\nu_2}z\bar{z}, \xi]_\star.
\]
We remark that because $N$ is a complex 1-dim vector bundle,
$\hbar_2$ does not appear in the above curvature expression although
it does show up in general in the star-product.

Let $\calw$ be the Weyl algebra (with coefficient in
$\complex((\hbar_1))$) bundle associated to the symplectic form
$\omega^0$ on $M^\gamma_2$. Following Fedosov's method
\cite{fe:book} and Kravchenko's modification \cite{Kr}, we construct
a flat connection $D$ on the associated bundle
\[
\wedge^\bullet T^*M^\gamma_2\otimes \calw\otimes \calv.
\]
We remark that $\calw$ is a bundle of algebras with respect to the
ring $\complex((\hbar_1))$, and $\calv$ is an algebra with respect
to the ring $\complex((\hbar_1))((\hbar_2))$. The tensor product
between $\calw$ and $\calv$ is taken over the ring
$C^\infty(M)((\hbar_1))$. Though our construction is essentially a
repetition of the ones in \cite{Kr}, since the Dunkl-Weyl algebra is
a new ingredient, we recall the construction of the flat connection
on $\wedge^\bullet T^*M^\gamma_2\otimes \calw\otimes \calv$ briefly.

Let $\nabla^T$ be a symplectic connection on $TM^\gamma_2$ with
respect to the symplectic form $\omega_0$. If $\Gamma_{ij}^k$ be the
Christoeffel symbol associated to the connection $\nabla^T$ on
$TM^\gamma_2$, then
\[
\partial^T \eta =d\eta+\frac{i}{2\hbar_1}[\omega^0_{il}\Gamma^l_{jk}y^iy^jdx^k,
\eta]_\ast,\ \eta\in \calw
\]
defines a connection on the Weyl algebra bundle $\calw$, where $y^i,
i=1,\dots 2n-2$ are coordinates along the fiber direction of
$TM^\gamma_2$ and $[\ ,\ ]_\ast$ is the commutator with respect to
the star-product $\ast$ on $\calw$. Accordingly,
$\partial:=\partial^T\otimes 1+1\otimes \partial^N$ defines a
connection on the bundle $\calw\otimes \calv$. It is a
straightforward computation to find
\[
\partial^2 a=\frac{i}{\hbar_1}[R^T\otimes 1+1\otimes R^N,a],\ a\in
\Gamma^\infty(\calw\otimes \calv),
\]
where $R^T$ is the curvature form of $\partial^T$.

Define $\delta:\wedge^\bullet T^*M^\gamma_2\otimes \calw\otimes
\calv\longrightarrow \wedge^\bullet T^*M^\gamma_2\otimes
\calw\otimes \calv$ by $\delta(a)=\sum_{i=1}^{2n-2}dx^i\partial
a/\partial y^i $. Then the same proof as \cite[Thm. 5.5]{Kr} proves
that there is a flat connection $D$ on $\wedge^\bullet
T^*M^\gamma_2\otimes \calw\otimes \calv$ of the form
\[
D=d+\frac{i}{\hbar_1}[\gamma,
\cdot]=\partial+\delta+\frac{i}{\hbar_1}[r,\cdot],
\]
where $r$ is an element in $T^*M^\gamma_2\otimes \calw\otimes
\calv$. The key point in the Kravchenko's proof of \cite[Theorem
5.5]{Kr} is that one has to modify the definition of the operator
$\delta$ to compensate the existence of the curvature form $R^N$ in
the expression of $\partial^2$ because $i/\hbar_1 R^N$ will
contribute an extra term of degree -1 in Fedosov's iteration
procedure of constructing a flat connection. This also applies to our construction with the following observation.

By the definition of the star-product on the Dunkl-Weyl algebra, one
can easily check that for $j\geq 2$, $C^0_{jl}$ and $C^1_{jl}$
vanishes on $z\bar{z}$. Therefore, we have  for an arbitrary $f\in
\mathbb{D}_2$, as $f$ is $\integers_2$ invariant,
{
\small
\[
\begin{split}
&[z\bar{z}, f]_\star\\
=&\hbar_1\Big( \{z\bar{z},f\}+\hbar_2\big(\frac{(z\bar{z}+z\bar{z})}{2\bar{z}}\frac{(f(z,-\bar{z})-f(-z,-\bar{z}))}{2z}-\frac{(f(z,\bar{z})-f(z,-\bar{z}))}{2\bar{z}}\frac{(-z\bar{z}-z\bar{z})}{2z}\big)\Big)\\
=&\hbar_1\{z\bar{z},f\},
\end{split}
\]
} where $\{\ .\ \}$ is the Poisson structure
$\{f,g\}=i(\partial_{\bar{z}}f\partial_zg-\partial_{\bar{z}}g\partial_zf)$.
This shows that on the bundle $\calv$ with the fiberwise algebra
isomorphic to $\mathbb{D}_2((\hbar_1))((\hbar_2))$, the curvature of
the connection $\partial^N$ which is equal to the star-commutator of
$R^N$ with respect to the product on $\mathbb{D}_2((\hbar_1,
\hbar_2))$ acts as same as the Poisson commutator associated to the
restriction of the symplectic form $\omega^N$ along the fiber
direction of $N$.  With this observation, we can repeat the exactly
same construction as those in Kravchenko's proof of \cite[Theorem
5.6]{Kr}. And we conclude that there is a flat connection $D$ on the
bundle $\wedge^\bullet T^*M^\gamma_2\otimes \calw\otimes \calv$
whose Weyl curvature is equal to
$\omega_c:=\widetilde{\omega}^0+\nu^2\big(\widetilde{\omega}^N+
R^Nz\bar{z} )\big)$, where $\widetilde{\omega}^0$ (and
$\widetilde{\omega}^N$) is the pullback of the symplectic form
$\omega^0$ (and $\omega^N$) onto $N$ (via the connection
$\nabla^N$), and $\nu$ is a sufficiently small real number. We have
chosen to use the subindex $c$ to stand for the name ``weak coupling
form" (\cite[Theorem 3.5]{Kr}). We point out that $\omega_c$ may not
be non-degenerate on the whole normal bundle $N$, but since
$M^\gamma$ is compact, $\omega_c$ does define a symplectic structure
on an $\epsilon$ neighborhood $N_\epsilon$ of the zero section in
$N$ with a sufficiently small $\epsilon$. To have the construction
in this section work, we need to restrict the algebra ${\mathbb
D}_2((\hbar_1))((\hbar_2))$ to an open ball with a sufficiently small
radius. Then all the above constructions in this subsection easily
generalize to define a quantization on $N_\epsilon$ as the product
on ${\mathbb D}_2((\hbar_1))((\hbar_2))$ is $\gamma$-local.

\subsection{Quantization of punctured disk bundle}
With the above flat connection $D$ on $\calw\otimes \calv$, we
consider flat sections with respect to the connection $D$. The space
$\cala_D$ of flat sections is isomorphic $C^\infty(M^\gamma_2,
\calv)$ as a vector space. Furthermore, including the isomorphism
between $\mathbb{D}_2((\hbar_1))((\hbar_2))$ with
$\operatorname{Poly}(\mathbb{R}^2)^{\integers_2}((\hbar_1))((\hbar_2))$ as vector
spaces, we conclude that $\cala_D$ is isomorphic to the space of
functions on $M^\gamma_2$ with value in the associated bundle
$P\otimes_{U(1)}\operatorname{Poly}(\mathbb{R}^2)^{\integers_2}((\hbar_1))((\hbar_2))$.
The later space can be viewed as the space
$C^\infty(N)^{\integers_2}((\hbar_1))((\hbar_2))$ of
$\integers_2$-invariant smooth functions on $N$.  The identification
between $C^\infty(N)^{\integers_2}((\hbar_1))((\hbar_2))$ and $\cala_D$
equips $C^\infty(N)^{\integers_2}((\hbar_1))((\hbar_2))$ with a new
associative product, which is a deformation of the standard
commutative product on
$C^\infty(N)^{\integers_2}((\hbar_1))((\hbar_2))$.

Via the exponential map with respect to some $\integers_2$ invariant
metric on $N$, we can identify a tubular neighborhood $B_\epsilon$
of $M^\gamma_2$ in $M$ with an $\epsilon$ neighborhood $N_\epsilon$
of the zero section in $N$ for some $\epsilon>0$. Furthermore, we
observe that the pullback of the symplectic form $\omega$ on $M$
defines a symplectic form $\omega$ on $N$. Since the restrictions of
$\omega$ and $\omega_c$ on $M^\gamma$ both are $\omega^0$, by
Moser's theorem (\cite[Theorem 7.4]{cannas}), there are
$\integers_2$ invariant neighborhoods $U_1$ and $U_2$ of $M^\gamma$
in $N$, and a $\integers_2$ equivariant diffeomorphism $\phi:U_1\to
U_2$ such that $\phi|_{M^\gamma}=id$, and $\phi^*\omega=\omega_c$.
Since $M$ is compact, $M^\gamma$ is also compact. So we can even
shrink $U_1$ and $U_2$ properly to make $U_1$ an $\epsilon$
neighborhood $N_\epsilon$ of the zero section in $N$. In the
following constructions, we always assume that we have used this
$\phi$ to identify the symplectic forms $\omega_c$ and $\omega$.

As was discussed in Theorem \ref{thm:symbol-cal}, the star-product
on $\mathbb{D}_2((\hbar_1))((\hbar_2))$ is $\gamma$-local. Furthermore,
the product on the standard Weyl algebra is also local. These
locality results imply that the deformed product on
$C^\infty(N)^{\integers_2}((\hbar_1))((\hbar_2))$ is $\gamma$-local,
and therefore, we are allowed to restrict
$(C^\infty(N)^{\integers_2}((\hbar_1))((\hbar_2)),\star)$ to the
$\epsilon$ neighborhood $N_\epsilon$ of the zero section in $N$,
which defines an associative deformation of $\integers_2$-invariant
functions on $N_\epsilon$. Finally, pushing forward along the
exponential map, we obtain a deformation of $\integers_2$-invariant
smooth functions on $B_\epsilon$, namely
$(C^\infty(B_\epsilon)^{\integers_2}((\hbar_1))((\hbar_2)), \star)$.

We next look at the space $B^*_\epsilon:=B_{\epsilon}-M^\gamma_2$ of
punctured neighborhood, which is diffeomorphic to
$N_\epsilon-M^\gamma_2$, the punctured disk bundle via the
exponential map. As the star product on the Dunkl-Weyl algebra
$\mathbb{D}_2$ is $\gamma$-local, we can restrict the algebra
$(C^\infty(B_\epsilon)^{\integers_2}((\hbar_1))((\hbar_2)), \star)$  to
the punctured neighborhood, which is denoted by
\[
(C^\infty(B^*_\epsilon)^{\integers_2}((\hbar_1))((\hbar_2)), \star).
\]
We observe that as the $\integers_2$ action on $B^*_\epsilon$ is
free, the space of $\integers_2$-invariant functions on
$B^*_\epsilon$ can be identified with the space of functions on the
quotient $B^*_\epsilon/\integers_2$, which is a smooth manifold.
Hence the algebra $(C^\infty(B^*_\epsilon)^{\integers_2}((\hbar_1,
\hbar_2)), \star)$ can be viewed as a deformation quantization of
the quotient $B^*_\epsilon/\integers_2$. On the other hand, the
$\integers_2$-invariant symplectic form $\omega$ on $M$ restricts to
define a symplectic form on the quotient $B^*_\epsilon/\integers_2$.
As $B^*_\epsilon/\integers_2$ is a smooth symplectic manifold, one
can apply the standard Fedosov construction of a deformation
quantization on $B^*_\epsilon/\integers_2$, and therefore obtain an
associative algebra $(C^\infty(B^*_\epsilon)^{\integers_2}((\hbar_1))((
\hbar_2)), \star_{\text{F}})$ with the characteristic class equal to
$\omega$.

\begin{proposition}\label{prop:quant-punctured-disk}
The algebras $(C^\infty(B^*_\epsilon)^{\integers_2}((\hbar_1))((
\hbar_2)), \star)$ and
$(C^\infty(B^*_\epsilon)^{\integers_2}((\hbar_1))((\hbar_2)),
\star_{\text{F}})$ are isomorphic.
\end{proposition}
\begin{proof}
We consider an intermediate algebra to relate the above two
algebras. We look at the normal bundle $N$ over $M^\gamma_2$. The
restriction of the symplectic form $\omega$ to each fiber makes $N$
into a symplectic vector bundle. We consider the associated Weyl
algebra bundle to $\calv_{\text{W}}$ by
$\calv_{\text{W}}:=P\times_{U(1)}\mathbb{W}_2^{\integers_2}$, where
$\mathbb{W}_2$ is the tensor of the Weyl algebra $W_2$ (with coefficient
in $\complex((\hbar_1))$) on $\mathbb{R}^2$ with
$\complex((\hbar_2))$ and $\mathbb{W}_2^{\integers_2}$ is the
$\integers_2$ invariant subalgebra of $\mathbb{W}_2$. Similar to
what we have done in Section \ref{sec:dunkl-weyl-alg}, we can
construct a flat connection $D_{\text{W}}$ on the bundle
$\wedge^\bullet T^*M^\gamma_2\otimes \calw\otimes \calv_{\text{W}}$.
The space of flat sections with respect to the flat connection
$D_{\text{W}}$ is isomorphic to $C^\infty(N)^{\integers_2}((\hbar_1,
\hbar_2))$. Therefore we obtain an associative algebra
$(C^\infty(N)^{\integers_2}((\hbar_1))((\hbar_2)), \star_{\text{W}})$
as a deformation quantization of
$C^\infty(N)^{\integers_2}((\hbar_1))((\hbar_2))$. As the Weyl algebra
$\mathbb{W}_2$ has a local product, the algebra
$(C^\infty(N)^{\integers_2}((\hbar_1))((\hbar_2)), \star_{\text{W}})$
restricts to the punctured disk bundle $N_\epsilon-M^\gamma_2$. And
via the exponential map with respect to the $\integers_2$ invariant
riemannian metric, $(C^\infty(N)^{\integers_2}((\hbar_1))((\hbar_2)),
\star_{\text{W}})$ restricts to define a deformation quantization of
the punctured tubular neighborhood,
$(C^\infty(B^*_\epsilon)^{\integers_2}((\hbar_1))((\hbar_2)),
\star_{\text{W}})$.

Now we compare the two products $\star_{\text{W}}$ and $\star_{\text{F}}$ on 
$C^\infty(B^*_\epsilon)^{\integers_2}((\hbar_1))((\hbar_2))$. According to \cite[Sec. 5]{fe:g-index} and
\cite[Thm 5.6]{Kr}, these two algebras are isomorphic as they have
the same characteristic class $\omega$. To compare the algebra
$(C^\infty(B^*_\epsilon)^{\integers_2}((\hbar_1))((\hbar_2)), \star)$
with $(C^\infty(B^*_\epsilon)^{\integers_2}((\hbar_1))((\hbar_2)),
\star_{\text{W}})$, we see that the procedure to obtain these two
algebras are different only at one step, where $\calv$ is $P\times
_{U(1)}\mathbb{D}_2((\hbar_1))((\hbar_2))$ and $\calv_{\text{W}}$ is
$P\times_{U(1)}\mathbb{W}_2^{\integers_2}$. We have pointed out that
the product on $\mathbb{D}_2((\hbar_1))((\hbar_2))$ is $\gamma$-local
and the product on $\mathbb{W}^{\integers_2}_2$ is local. Therefore,
the restrictions of $(C^\infty(N)^{\integers_2}((\hbar_1))((\hbar_2)),
\star)$ and $(C^\infty(N)^{\integers_2}((\hbar_1))((\hbar_2)),
\star_{\text{W}})$ to $N_\epsilon-M^\gamma_2$ can be constructed via
the flat connections $D$ and $D_{\text{W}}$ on the bundle
$\wedge^\bullet T^*M^\gamma_2\otimes \calw\otimes
(P\times_{U(1)}\mathbb{D}_2((\hbar_1))((\hbar_2)) |_{D^*_\epsilon})$
and $\wedge^\bullet T^*M^\gamma_2\otimes \calw\otimes
(P\times_{U(1)}\mathbb{W}_2^{\integers_2} |_{D^*_\epsilon})$, where
$D^*_\epsilon$ is the puncture disk of radius $\epsilon$ in
$\mathbb{R}^2$.

The algebras  $\mathbb{D}_2((\hbar_1))((\hbar_2)) |_{D^*_\epsilon}$ and
$\mathbb{W}_2^{\integers_2} |_{D^*_\epsilon}$ are both deformation
quantization of the punctured disk $D^*_\epsilon/\integers_2$. In
fact, if we look at the algebra $\mathbb{D}_2((\hbar_1))((\hbar_2))
|_{D^*_\epsilon}$ more carefully, we notice that the product of this
algebra is an expression of $f,g\in
C^\infty(D^*_\epsilon)^{\integers_2}$ of power series $\hbar_1$ and
$\hbar_2$. In particular, if we look at $f\star g$ as a formal power
series of $\hbar_2$, the 0-th power term is exactly the product on
$\mathbb{W}_2^{\integers_2} |_{D^*_\epsilon}$. Therefore, we can
view $\mathbb{D}_2((\hbar_1))((\hbar_2)) |_{D^*_\epsilon}$ as a formal
deformation quantization of the algebra $W_2^{\integers_2}
|_{D^*_\epsilon}$, where $W_2^{\integers_2}$ is the subspace of
$\integers_2$ invariant elements in the Weyl algebra $W_2$ (with
coefficient in $\complex ((\hbar_1))$). As we can identify
$W_2^{\integers_2} |_{D^*_\epsilon}$ as a deformation quantization
of the quotient $D^*_\epsilon/\integers_2$, its Hochschild
cohomology of $W_2^{\integers_2} |_{D^*_\epsilon}$ can be computed
using the result of \cite{PflPosTanTse}. In particular, the second
Hochschild cohomology of $W_2^{\integers_2} |_{D^*_\epsilon}$ is
equal to the degree 2 de Rham cohomology of
$D^*_\epsilon/\integers_2$ with coefficient in
$\complex((\hbar_1))$. As $D^*_\epsilon/\integers_2$ is homotopic to
a circle, its degree 2 de Rham cohomology is zero. This implies that
$\mathbb{D}_2((\hbar_1))((\hbar_2)) |_{D^*_\epsilon}$ must be a trivial
deformation of $W_2^{\integers_2} |_{D^*_\epsilon}$. Furthermore, as
$U(1)$ is compact, by the standard averaging trick, we can obtain a
$U(1)$ equivariant isomorphism from $\mathbb{D}_2((\hbar_1))((\hbar_2))
|_{D^*_\epsilon}$ to $\mathbb{W}_2^{\integers_2}
|_{D^*_\epsilon}=W_2^{\integers_2}[\hbar_2^{-1},
\hbar_2]]_{D^*_\epsilon}$.

We look at the above construction of the isomorphism between
$\mathbb{D}_2((\hbar_1))((\hbar_2)) |_{D^*_\epsilon}$ and
$\mathbb{W}_2^{\integers_2}
|_{D^*_\epsilon}=W_2^{\integers_2}[\hbar_2^{-1},
\hbar_2]]|_{D^*_\epsilon}$ more carefully. If we denote
$\widetilde{W}_2$ to be the Weyl algebra on $\reals^2$ with
coefficient in $\complex[[\hbar_1]]$, the algebra
$\mathbb{D}_2[[\hbar_1,\hbar_2]] |_{D^*_\epsilon}$ can also be
viewed as a deformation of the algebra
$\widetilde{W}_2^{\integers_2}[[\hbar_2]] |_{D^*_\epsilon}$. By
Theorem \ref{thm:deformation-main}, (2) and its explanation in
Section \ref{subsec:bilinear-operators}, we have the property that
for every $i\geq 1$, the $\hbar_2^i$ term in the deformation
$\mathbb{D}_2[[\hbar_1,\hbar_2]] |_{D^*_\epsilon}$ takes value in
$\hbar_1\widetilde{W}_2^{\integers_2}[[\hbar_2]] |_{D^*_\epsilon}$.
We can use spectral sequence associated the $\hbar_1$-filtration to
compute the Hochschild cohomology of the algebra
$\widetilde{W}_2^{\integers_2}|_{D^*_\epsilon}$. The spectral
sequence degenerates at $E_2$ with
$E_2^{0,2}=\Gamma\Big(\wedge^2T\big(D^*_{\epsilon}/\integers_2\big)\Big)$
and $E_2^{p,q}=\{0\}$, for $p+q=2, p\geq 1$.
$HH^2(\widetilde{W}_2^{\integers_2}|_{D^*_\epsilon},
\widetilde{W}_2^{\integers_2}|_{D^*_\epsilon})$ naturally projects
onto $E_\infty^{0,2}\cong E^{0,2}_2$, which in our case is an
isomorphism. From this computation, we conclude that as the cocycles
on $\widetilde{W}_2^{\integers_2}|_{D^*_\epsilon}$ in the
deformation $\mathbb{D}_2[[\hbar_1,\hbar_2]] |_{D^*_\epsilon}$ take
value in $\hbar_1\widetilde{W}_2^{\integers_2}[[\hbar_2]]
|_{D^*_\epsilon}$, they must vanish in the cohomology along the
projection from $HH^2(\widetilde{W}_2^{\integers_2}|_{D^*_\epsilon},
\widetilde{W}_2^{\integers_2}|_{D^*_\epsilon})$ to $E^{0,2}_2$.
Therefore all the above cocycles must be coboundaries. This
observation allows us to choose a $U(1)$-equivariant isomorphism
between $\mathbb{D}_2[[\hbar_1,\hbar_2]] |_{D^*_\epsilon}$ and
$\widetilde{W}_2^{\integers_2}[[\hbar_2]] |_{D^*_\epsilon}$, which
naturally extends to define a $U(1)$ equivariant isomorphism from
$\mathbb{D}_2((\hbar_1))((\hbar_2)) |_{D^*_\epsilon}$ to
$\mathbb{W}_2^{\integers_2}
|_{D^*_\epsilon}=W_2^{\integers_2}[\hbar_2^{-1},
\hbar_2]]_{D^*_\epsilon}$.

We notice that the action of the Lie algebra of $U(1)$ on $\mathbb{D}_2((\hbar_1))((\hbar_2)) |_{D^*_\epsilon}$ (and
$\mathbb{W}_2^{\integers_2}
|_{D^*_\epsilon}=W_2^{\integers_2}[\hbar_2^{-1},
\hbar_2]]_{D^*_\epsilon}$) can be expressed as the commutator operator with respect to the function $z\bar{z}$ in  $\mathbb{D}_2((\hbar_1))((\hbar_2)) |_{D^*_\epsilon}$ (and
$\mathbb{W}_2^{\integers_2}
|_{D^*_\epsilon}=W_2^{\integers_2}[\hbar_2^{-1},
\hbar_2]]_{D^*_\epsilon}$). The $U(1)$ equivariance property implies that the isomorphism between $\mathbb{D}_2((\hbar_1))((\hbar_2)) |_{D^*_\epsilon}$ and
$\mathbb{W}_2^{\integers_2}
|_{D^*_\epsilon}=W_2^{\integers_2}[\hbar_2^{-1},
\hbar_2]]_{D^*_\epsilon}$ identifies $z\bar{z}$ modulo center elements.

With the above $U(1)$ equivariant isomorphism between the algebras
on each fiber, we have an natural isomorphism of bundles $\calv$ and
$\calv_{\text{F}}$ which accordingly identifies the connections
$\partial_N$ on the corresponding bundles. Noticing that Fedosov's
construction of flat connection is canonical with respect to the
choice of a symplectic connection, we can easily check that the
construction of flat connections $D$ and $D_{\text{W}}$ are actually
compatible with respect to this isomorphism of bundles. (We point
out that we may have to adjust $z\bar{z}$ in the construction of the
connection $\partial^N$ by a center element, which is in
$\complex[[\hbar_1, \hbar_2]]$, due to the identification. But this
change does not affect the whole construction of the algebras.)
Hence, we can conclude that there is an isomorphism between
$(C^\infty(B^*_\epsilon)^{\integers_2}((\hbar_1))((\hbar_2)), \star)$
and $(C^\infty(B^*_\epsilon)^{\integers_2}((\hbar_1))((\hbar_2)),
\star_{\text{W}})$ as flat sections of $D$ and $D_{\text{W}}$. We
remark that due the fact that isomorphism

We conclude that $(C^\infty(B^*_\epsilon)^{\integers_2}((\hbar_1,
\hbar_2)), \star)$ is isomorphic to
$(C^\infty(B^*_\epsilon)^{\integers_2}((\hbar_1))((\hbar_2)),
\star_{\text{W}})$ and therefore is isomorphic to
$(C^\infty(B^*_\epsilon)^{\integers_2}((\hbar_1))((\hbar_2)),
\star_{\text{F}})$.
\end{proof}

\subsection{Global algebra}

In this subsection, we construct the algebra promised at the
beginning of this section, which is a deformation of
$A^{((\hbar_1))}_{M/\integers_2}$.

Recall that $M^\gamma_2$ is a disjoint union of fixed point
submanifolds of $M$ which are of codimension 2. Fix a
$\integers_2$-invariant almost complex structure on $M$, which also
defines a $\integers_2$ invariant metric on $M$. We choose a
sufficiently small $\epsilon$ such that the $\epsilon$ tubular
neighborhood of each component of $M^\gamma_2$ in $M$ does not
intersect with each other. We use $B_\epsilon$ to denote the
disjoint union of the $\epsilon$ tubular neighborhood of each
component of $M^\gamma_2$. Furthermore, we denote $M^-$ to be the
open complement $M-M^\gamma_2$ to the closed subset $M^\gamma_2$. In
this way, we have the orbifold as a union of two open subsets
$B_\epsilon/\integers_2$ and $M^-/\integers_2$, and the intersection
of these two open sets is $B^*_\epsilon/\integers_2$, the $\epsilon$
punctured neighborhood of $M^\gamma$ in $M/\integers_2$.

We construct an algebra $\mathfrak
{A}^{((\hbar_1))((\hbar_2))}_{M/\integers_2}$ as follows. On
$B_\epsilon/\integers_2$, this algebra is isomorphic to
$(C^\infty(B_\epsilon)^{\integers_2}((\hbar_1))((\hbar_2)), \star)$
which, via the exponential map, can be identified as the restriction
to the bundle $N_\epsilon$ of the space of flat sections of the
Dunkl-Weyl algebra introduced in Section \ref{sec:dunkl-weyl-alg}.
On $M^-/\integers_2$, the algebra $\mathfrak
{A}^{((\hbar_1))((\hbar_2))}_{M/\integers_2}$ is isomorphic to the
Fedosov quantization\footnote{The algebra
$\cala^{((\hbar_1))((\hbar_2))}_{M^-}$ is defined to
$\cala^{((\hbar_1))}_{M^-}\otimes \complex((\hbar_2))$.}
$(\cala^{((\hbar_1))((\hbar_2))}_{M^-})^{\integers_2}$ of
$M^-/\integers_2$ with the Weyl curvature being $\omega$. By
Proposition \ref{prop:quant-punctured-disk}, the restriction of
$(C^\infty(B_\epsilon)^{\integers_2}((\hbar_1))((\hbar_2)), \star)$ to
$B^*_{\epsilon}/\integers=M^-/\integers_2\cap
B_\epsilon/\integers_2$ is isomorphic to
$(C^\infty(B^*_\epsilon)^{\integers_2}((\hbar_1))((\hbar_2)),
\star_{\text{F}})$, which is the restriction of
$(\cala^{((\hbar_1))((\hbar_2))}_{M^-})^{\integers_2}$ to
$B^*_\epsilon/\integers_2$. We define
$\mathfrak{A}^{((\hbar_1))((\hbar_2))}_{M/\integers_2}$ to be the
algebra defined by gluing
$(C^\infty(B_\epsilon)^{\integers_2}((\hbar_1))((\hbar_2)), \star)$
and $(\cala^{((\hbar_1))((\hbar_2))}_{M^-})^{\integers_2}$ via the
isomorphism on $B^*_\epsilon/\integers_2$.

We summarize the above construction into the following theorem.
\begin{theorem}\label{thm:deformation-main}The algebra
$\mathfrak{A}^{((\hbar_1))((\hbar_2))}_{M/\integers_2}$ is a nontrivial
deformation of the algebra $\cala^{((\hbar_1))}_{M/\integers_2}$.
\end{theorem}
\begin{proof}
We look at the product on the algebra
$\mathfrak{A}^{\hbar_1,\hbar_2}$. If $f$ and $g$ are two elements of
$C^\infty(M/\integers_2)$, $f\star g$ can be written as a formal
power series of $\hbar_2$. From the construction in Section
\ref{sec:dunkl-weyl-alg}, it is not difficult to see that the
$\hbar_2^0$ component is exactly the product of the algebra
$\cala^{((\hbar_1))}_{M/\integers_2}$. Furthermore, from the local
computation on $B_\epsilon$, we can see that as
$\mathbb{D}_2((\hbar_1))((\hbar_2))$ is a nontrivial deformation of the
invariant Weyl algebra $\mathbb{W}_2^{\integers_2}$, the algebra
$\mathfrak{A}^{((\hbar_1))((\hbar_2))}_{M/\integers_2}$ is a nontrivial
deformation.
\end{proof}
\begin{remark}
In the construction of the algebra
$\mathfrak{A}^{((\hbar_1))((\hbar_2))}$, we have chosen an $\epsilon$
neighborhood $B_\epsilon$ of $M^\gamma_2$. We point out that
different choices of $\epsilon$ give rise to isomorphic algebras
$\mathfrak{A}^{((\hbar_1))((\hbar_2))}$. It is not hard to check the
$U(1)$ equivariant isomorphism from $\mathbb{D}_2((\hbar_1))((\hbar_2))
|_{D^*_\epsilon}$ to $\mathbb{W}_2^{\integers_2} |_{D^*_\epsilon}$
constructed in Proposition \ref{prop:quant-punctured-disk} can be
made to be compatible with the restriction map to $B_{\epsilon'}$
with $\epsilon'<\epsilon$, as the operators appearing in the
isomorphism are all $\gamma$-local operators. This compatibility
with respect to the restriction map assures that the outcome algebra
$\mathfrak{A}^{((\hbar_1))((\hbar_2))}$ are all isomorphic for
different choices of $\epsilon$.

We have used an almost complex structure and therefore a compatible
riemannian metric to identify the $\epsilon$ neighborhood
$B_\epsilon$ with the $\epsilon$ neighborhood $N_\epsilon$ of the
zero section of $N$. Our algebra $\mathfrak{A}^{((\hbar_1,
\hbar_2))}$ does seem to depend on the choice of almost complex
structures since we are taking the normal ordering of the operator
symbol calculus in Definition \ref{dfn:symbol} and also in
Definition \ref{dfn:deformation} of the Dunkl-Weyl algebra
$\mathbb{D}_2((\hbar_1))((\hbar_2))$. The analogous well-known
phenomena is that wick and anti-wick deformation quantization of an
almost K\"ahler manifold depends on the choices of almost complex
structures. We plan to discuss this dependence of almost complex
structures in the future.
\end{remark}

\begin{remark}\label{rmk:conjecture}We explain how far we are away from a full proof of
the Dolgushev-Etingof conjecture in the case of a $\integers_2$
orbifold. In this section, we have constructed a deformation of the
algebra $A^{((\hbar_1))}_{M/\integers_2}$ along the direction of the
union $M^\gamma_2$ of all codimension 2 components in the inertia
orbifold $\widetilde{M/\integers_2}$. Furthermore, we observe that
our constructions are local with respect to every connected
component $M^\gamma_2$. Such an observation allows us to construct a
deformation of $A^{((\hbar_1))}_{M/\integers_2}$ for every connected
component of $M^\gamma_2$. If we associate a formal parameter $c_i$
for every component of $M^\gamma_2$, we actually have constructed a
universal deformation of $A^{((\hbar_1))}_{M/\integers_2}$
parametrized by codimension 2 components in $\widetilde{M/\integers_2}$. This is
the main part of the Dolgushev-Etingof conjecture \cite{D-E}.

We are not able to prove the full conjecture of Dolgushev-Etingof
for $\integers_2$ orbifolds because in our construction, in
particular the proof of Proposition \ref{prop:quant-punctured-disk},
we have used crucially two extra assumptions. One is that we have
assumed the characteristic class of
$A^{((\hbar_1))}_{M/\integers_2}$ is $\omega$, the other is that the
parameter $\hbar_2$ is formal. However, the Dolgushev-Etingof
conjecture \cite{D-E} does not require these two assumptions. The
first assumption allows us to compare the Fedosov quantizations of the
normal bundle $N$ of the fixed point submanifold $M^\gamma_2$ and a
tubular neighborhood $B_\epsilon$. If we change the characteristic
class $\omega$ by a class $t$ in $H^2(M/\integers_2)((\hbar_1))$, it
is hard to realize the information of $t$ on the normal bundle $N$,
which prevents us from comparing the corresponding Fedosov
quantizations. The second assumption that $\hbar_2$ is formal allows
us to apply homological algebra arguments to show that ${\mathbb
{D}}_2((\hbar_1))((\hbar_2))|_{D^*_\epsilon}$ is isomorphic to
$\mathbb{W}_2^{\integers_2}|_{D^*_\epsilon}$. If we are able to
prove that the constructed isomorphism in power series of $\hbar_2$
is convergent, then we can actually allow $\hbar_2$ to be a number
in $\complex$. We do not have solutions to avoid these two
assumptions now, and hope to address these problems in a future
publication.
\end{remark}

We finally remark that in this section, we have been working with the
global quotient orbifold, the quotient of a symplectic manifold $M$ by a
$\integers_2$ action. Our construction does generalize for general
orbifolds which is locally either diffeomorphic to $\reals^n$ or the
quotient of $\reals^n$ by the linear $\integers_2$ action.
\section{Proof Theorem \ref{thm:symbol-cal}}
\label{sec:proof-thm}

We will prove Theorem \ref{thm:symbol-cal} in 2 steps. In the first
step, we work with Dunkl pseudodifferential operators of the forms
introduced in Definition \ref{dfn:pseud-op} to derive an asymptotic
expansion of the symbol of the product of two operators. The
asymptotic expansion of the symbol we obtain in Step I may contain
a sum of infinitely many terms in a fixed symbol class. In the
second step, we will rewrite the asymptotic expansion of the symbol
obtained in Step I into the expressions introduced in Section
\ref{subsec:bilinear-operators}.
\subsection{Step I}
Let $a$, $b$ be two polynomials on $\mathbb{R}^2$. To compute the
asymptotic expansion of $Op_k(a)\circ Op_k(b)$ we need to study the
following integral
\[
\int_{\mathbb{R}}\int_{\mathbb{R}}\int_{\mathbb{R}}d\mu_k(p)d\mu_k(y)d\mu_k(p_1)E_{k}(x,ip)a(x,p)E_k(y,-ip)E_k(y,
ip_1)b(y,p_1)E_{k}(z,-ip_1).
\]
Since $a(x,p)$ is a polynomial, we take the Taylor expansion of
$a(x,p)$ with respect to $p$, i.e.
$a(x,p)=\sum_{\alpha}p^\alpha/\alpha!\partial_p^\alpha a(x,0)$.

We insert the Taylor expansion into the above equation of
$Op_k(a)\circ Op_k(b)$.
\[
\begin{split}
\int_{\mathbb{R}}\int_{\mathbb{R}}\int_{\mathbb{R}}&d\mu_k(p)d\mu_k(y)d\mu_k(p_1)\sum_\alpha
E_{k}(x,ip)\frac{p^\alpha}{\alpha^!}\partial_p^\alpha
a(x,0)E_k(y,-ip)\\
&\qquad\qquad E_k(y, ip_1)b(y,p_1)E_{k}(z,-ip_1).
\end{split}
\]
Recall that when applying the variable $y$, we have
$T_k(E_k(y,-ip))=(-ip)E_{k}(y,-ip)$. Hence we can replace
$pE_{k}(y,-ip)$ by $iT_k(E_k(y,-ip))$ in the above integral, and
obtain
\[
\begin{split}
\int_{\mathbb{R}}\int_{\mathbb{R}}\int_{\mathbb{R}}&d\mu_k(p)d\mu_k(y)d\mu_k(p_1)\sum_\alpha
E_{k}(x,ip)\frac{p^{\alpha-1}}{\alpha^!}\partial_p^\alpha
a(x,0)iT_k(E_k(y,-ip))\\
&\qquad\qquad E_k(y, ip_1)b(y,p_1)E_{k}(z,-ip_1).
\end{split}
\]
As $T_k$ is a skew adjoint operator on $L^2_k(\mathbb{R})$,  we can
rewrite the above equation as
\[
\begin{split}
\int_{\mathbb{R}}\int_{\mathbb{R}}\int_{\mathbb{R}}&d\mu_k(p)d\mu_k(y)d\mu_k(p_1)\sum_\alpha
E_{k}(x,ip)\frac{p^{\alpha-1}}{\alpha^!}\partial_p^\alpha
a(x,0)E_k(y,-ip)\\
&\qquad\qquad (-i)T_k\big(E_k(y, ip_1)b(y,p_1)\big)E_{k}(z,-ip_1).
\end{split}
\]
We apply
\[
\begin{split}
&T_k\big(E_k(y,
ip_1)b(y,p_1)\big)\\
=&ip_1E_k(y,ip_1)b(y,p_1)+E_k(y,ip_1)\partial_yb(y,p_1)+E_k(-y,ip_1)k\tilde{\partial}_yb(y,p_1)
\end{split}
\]
to the above equation, and obtain
\[
\begin{split}
&\int_{\mathbb{R}}\int_{\mathbb{R}}\int_{\mathbb{R}}d\mu_k(p)d\mu_k(y)d\mu_k(p_1)\sum_\alpha
E_{k}(x,ip)\frac{p^{\alpha-1}}{\alpha^!}\partial_p^\alpha
a(x,0)E_k(y,-ip)\\
&\qquad \big(p_1E_k(y,
ip_1)b(y,p_1)-iE_k(y,ip_1)\partial_yb(y,p_1)-E_k(-y,ip_1)k\tilde{\partial}_yb(y,p_1))\big)E_{k}(z,-ip_1).
\end{split}
\]
Substituting $p_1$ by $-p_1$, we obtain the following expression
\[
\begin{split}
\int_{\mathbb{R}}\int_{\mathbb{R}}\int_{\mathbb{R}}&d\mu_k(p)d\mu_k(y)d\mu_k(p_1)\sum_\alpha
E_{k}(x,ip)\frac{p^{\alpha-1}}{\alpha^!}\partial_p^\alpha
a(x,0)E_k(y,-ip)\\
&\qquad\qquad
E_k(y,ip_1)\big((p_1-i\partial_y)b(y,p_1)-ik\tilde{\partial}_yb(y,-p_1)\hat{\gamma})\big)E_{k}(z,-ip_1),
\end{split}
\]
where $\hat{\gamma}$ is an operator on variable  $z$ changing $z$
to $-z$. In summary, we have seen above that an extra variable $p$
on $a(x,p)$ in the integral of $Op_k(a)\circ Op_k(b)$ is equivalent
to apply $p_1-i\partial_y-ik\sigma_2\tilde{\partial}\hat{\gamma}$ on
$b$, where $\sigma_2$ is mapping $b(x,p)$ to $b(x,-p)$.

By induction with respect to the power $\alpha$, we have the following expression
\[
\begin{split}
&\int_{\mathbb{R}}\int_{\mathbb{R}}\int_{\mathbb{R}}d\mu_k(p)d\mu_k(y)d\mu_k(p_1)E_{k}(x,ip)a(x,p)E_k(y,-ip)E_k(y,
ip_1)\\
&\qquad\qquad b(y,p_1)E_{k}(z,-ip_1)\\
=&\int_{\mathbb{R}}\int_{\mathbb{R}}\int_{\mathbb{R}}d\mu_k(p)d\mu_k(y)d\mu_k(p_1)E_{k}(x,ip)E_k(y,-ip)E_k(y,
ip_1)\\
&\qquad\qquad \sum_\alpha
\frac{1}{\alpha!}\partial_p^{\alpha}a(x,0)[p_1-i\partial_y-ik\sigma_2\tilde{\partial}_y\hat{\gamma}]^\alpha b(y,p_1) E_{k}(z,-ip_1).
\end{split}
\]
Integrating over variable $p$, we have the integral on the right
hand side equal to
\[
\begin{split}
&\int_{\mathbb{R}}\int_{\mathbb{R}}d\mu_k(y)d\mu_k(p)E_k(x,
ip_1)\\
&\qquad\qquad \sum_\alpha
\frac{1}{\alpha!}\partial_p^{\alpha}a(x,0)[p_1-i\partial_y-ik\sigma_2\tilde{\partial}_y\hat{\gamma}]^\alpha b(y,p_1)|_{y=x} E_{k}(z,-ip_1).
\end{split}
\]
Therefore, we conclude that
\begin{equation}\label{eq:expansion}
Op_k(a)\circ Op_k(b)=Op_k( \sum_\alpha
\frac{1}{\alpha!}\partial_p^{\alpha}a(x,0)[p_1-i\partial_y-ik\sigma_2\tilde{\partial}_y\hat{\gamma}]^\alpha b(y,p_1)|_{y=x} ).
\end{equation}
We remark that the above sum is finite as $a$ is a polynomial.
\subsection{Step II}
In this step,  we aim to understand the expansion formula
\[
\sum_\alpha
\frac{1}{\alpha!}\partial_p^{\alpha}a(x,0)[p_1-i\partial_y-ik\sigma_2\tilde{\partial}_y\hat{\gamma}]^\alpha b(y,p_1)|_{y=x}
\]
obtained in the previous subsection.

We look at the power
$[p_1-i\partial_y-ik\sigma_2\tilde{\partial}_y\hat{\gamma}]^\alpha$.
Define $A=-i\partial_y$,  and
$B=-ik\sigma_2\tilde{\partial}_y\hat{\gamma}$. We observe the
following commuting relations,
\[
Ap_1=p_1A,\qquad p_1B=-Bp_1.
\]
With these relations, we write $[p_1-i\partial_y-ik\sigma_2\tilde{\partial}_y\hat{\gamma}]^\alpha$ as
\[
\sum_{\nu\in P_{m,n}} c_{\nu}(-i)^{m+n}k^n p_1^{\alpha-m-n}B_\nu \sigma_2^n\hat{\gamma}^n,
\]
where $P_{m,n}$ is the set of solutions to Eq.
(\ref{eq:linear-partition}) introduced in Section
\ref{subsec:bilinear-operators}, and $B_\nu$ is the operator
introduced in Section \ref{subsec:bilinear-operators} as
compositions of $\partial_y$ and $\tilde{\partial}_y$, and $c_{\nu}$
is number determined by $\nu\in P_{m,n}$.

We study the number $c_{\nu }$ more carefully. $c_{\nu}$ is the number of the term
\[
(-i)^{m+n}k^n p_1^{\alpha-m-n}B_\nu \sigma_2^n\hat{\gamma}^n
\]
appearing in the expansion
$[p_1-i\partial_y-ik\sigma_2\tilde{\partial}_y\hat{\gamma}]^\alpha$.
When we write out the expansion of
$[p_1-i\partial_y-ik\sigma_2\tilde{\partial}_y\hat{\gamma}]^\alpha$,
it is a sum of monomials of the form
$p_1^{x_0}\partial_y^{\nu_0}\tilde{\partial}_yp_1^{x_1}\partial_y^{\nu_1}\tilde{\partial}_y\cdots
\tilde{\partial}_y
p_1^{x_n}\partial_y^{\nu_n}\sigma_2^n\hat{\gamma}^n$, where
$\nu=(\nu_0,\dots, \nu_n)$ is a fixed element of $P_{m,n}$, and
$x_0, \dots, x_n$ are nonnegative integers with
$x_0+\cdots+x_n=\alpha-m-n$. We remark that as $p_1$ commutes with
$\partial_y$, we do not need to count the relative positions between
$\partial_y$ and $p_1$.  Therefore, totally there are
\[
\prod_{j=0}^{n}\left(\begin{array}{c}y_i+x_i\\ x_i \end{array}\right)
\]
number of the term $p_1^{x_0}\partial_y^{\nu_0}\tilde{\partial}_yp_1^{x_1}\partial_y^{\nu_1}\tilde{\partial}_y\cdots \tilde{\partial}_y p_1^{x_n}\partial_y^{\nu_n}\sigma_2^n\hat{\gamma}^n$ in the power $[p_1-i\partial_y-ik\sigma_2\tilde{\partial}_y\hat{\gamma}]^\alpha$. Furthermore, as $\sigma_2$ changes the sign of $p_1$, when we move $\sigma_2$ to the right end, we need to count the change of signs.  Therefore, in front of the term $p_1^{x_0}\partial_y^{\nu_0}\tilde{\partial}_yp_1^{x_1}\partial_y^{\nu_1}\tilde{\partial}_y\cdots \tilde{\partial}_y p_1^{x_n}\partial_y^{\nu_n}\sigma_2^n\hat{\gamma}^n$, there should be a sign
\[
(-1)^{x_1+2x_2+\cdots+nx_n}.
\]
In summary, for $\nu\in P_{m,n}$, $c_\nu$ is equal to
\[
\sum_{x_0+\cdots +x_n=\alpha-m-n} \prod_{j=0}^{n}\left(\begin{array}{c}y_j+x_j\\ x_j \end{array}\right)(-1)^{jx_j}
\]
Separating $j$ from even to odd, we have $c_\nu$ equal to
\[
\begin{split}
&\sum_{x_0+\cdots +x_n=\alpha-m-n} \left[\prod_{j\ \text{is even}}\left(\begin{array}{c}y_j+x_j\\ x_j \end{array}\right)\right] \left[\prod_{j\ \text{is odd}}\left(\begin{array}{c}y_j+x_j\\ x_j \end{array}\right)(-1)^{x_i}\right]\\
=&\sum_{s+t=\alpha-m-n}\left[\sum_{x_0+x_2+\cdots
+x_{\text{even}}=s}\prod_{j\ \text{is
even}}\left(\begin{array}{c}y_j+x_j\\ x_j \end{array}\right)\right]
\left[\sum_{x_1+\cdots+x_{\text{odd}}=t}\prod_{j\
\text{odd}}\left(\begin{array}{c}y_j+x_j\\ x_j
\end{array}\right)(-1)^{x_i}\right]
\end{split}
\]
To evaluate the above number, we introduce the generating function $1/(1-x)^{s+1}$.
The Taylor expansion of $1/(1-x)^{1+s}$ and $1/(1+x)^{1+s}$ at $0$ with $|x|<1$ is
\[
\begin{split}
\frac{1}{(1-x)^{1+s}}&=\sum_{k=0}^\infty\left(\begin{array}{c}s+k\\ k\end{array}\right)x^k,\\
\frac{1}{(1+x)^{1+s}}&=\sum_{k=0}^\infty\left(\begin{array}{c}s+k\\ k\end{array}\right)(-1)^kx^k.
\end{split}
\]
In summary, if we denote $\Lambda_0=\nu_0+\sum_{\text{even i}}\nu_i$ and $\Lambda_1=\sum_{\text{odd i}}\nu_i$, then $c_\nu$ is the coefficient of the term $x^{\alpha-m-n}$ of the Taylor expansion of the following function
\begin{equation}\label{eq:int-c-nu}
\frac{1}{(1-x)^{\Lambda_0+n_0+1}(1+x)^{\Lambda_1+n_1}},
\end{equation}
where $n_0$ (and $n_1$) is the number of positive even (odd) numbers less than or equal to $n$.

We next consider the expansion
\[
\sum_\alpha
\frac{1}{\alpha!}\partial_p^{\alpha}a(x,0)[p_1-i\partial_y-ik\sigma_2\tilde{\partial}_y\hat{\gamma}]^\alpha b(y,p_1)|_{y=x}.
\]
By inserting the expansion of the power $[p_1-i\partial_y-ik\sigma_2\tilde{\partial}_y\hat{\gamma}]^\alpha$, we have the above expansion equal to
\[
\begin{split}
&\sum_{\alpha} \sum_{m,n,\nu}\frac{(-i)^{m+n}k^n}{\alpha!}\partial_p^\alpha a(x,0)c_\nu p_1^{\alpha-m-n}B_{\nu}\sigma_2^n(b)\hat{\gamma}^n\\
=&\sum_{m,n}(-i)^{m+n}k^n\big(\sum_{\nu\in P_{m,n}}\sum_{\alpha} \frac{c_\nu}{p_1^{m+n}}\frac{p_1^\alpha}{\alpha!}\partial_p^\alpha a(x,0)\big) B_{\nu}\sigma_2^n(b)\hat{\gamma}.
\end{split}
\]
We notice that the terms $B_{\nu}\sigma_2^n(b)\hat{\gamma}^k$ are
independent of $\alpha$ and $\nu$, then we are left to deal with $
\frac{c_\nu}{p_1^{m+n}}\frac{p_1^\alpha}{\alpha!}\partial_p^\alpha
a(x,0)$ for the sum over $\alpha$.

Considering the above interpretation of $c_\nu$, if we introduce an
auxiliary variable $t\in \complex-\{0\}$, then we have
$\sum_{\alpha}\frac{c_\nu}{p_1^{m+n}}\frac{p_1^\alpha}{\alpha!}\partial_p^\alpha
a(x,0)$ equal to the $t^0$ term of the product between
\[
\frac{t^{m+n}}{(1-tp_1)^{\Lambda_0+n_0+1}(1+tp_1)^{\Lambda_1+n_1}}
\]
and $a(x,1/t)$ for $|tp_1|<1$. We remark that by
$1/(1-tp_1)^{\Lambda_0+n_0+1}(1+tp_1)^{\Lambda+n_1}$ we really mean
the Taylor expansion with respect to variable $tp_1$ as $|p_1t|<1$,
and by $a(x,1/t)$ we mean the Taylor expansion of $a$ with respect
to the variable $1/t$. As $a$ is assumed to be a polynomial, its
Taylor expansion with respect to variable $1/t$ only has finitely
many terms, and $a(x,1/t)$ is an element in $\complex[x]((t))$. This
assures the product between $a(x,1/t)$ and
$t^{m+n}/(1-tp_1)^{\Lambda_0+n_0+1}(1+tp_1)^{\Lambda_1+n_1}$ well
defined, as a product of two Laurent series of variable $t$. In
conclusion, we conclude that
$\sum_{\alpha}\frac{c_\nu}{p_1^{m+n}}\frac{p_1^\alpha}{\alpha!}\partial_p^\alpha
a(x,0)$ is equal to the $t^0$ component of the product
$t^{m+n}a(x,1/t)1/(1-tp_1)^{\Lambda_0+n_0+1}(1+tp_1)^{\Lambda_1+n_1}$.

Now we relate the above explanation of
$\sum_{\alpha}\frac{c_\nu}{p_1^{m+n}}\frac{p_1^\alpha}{\alpha!}\partial_p^\alpha
a(x,0)$ with the operator $\Delta$ in Proposition \ref{prop:delta}.
Let $f$ be a polynomial of one variable. Then
$\Delta(f)(q_1,q_2)=(f(q_1)-f(q_2))/(q_1-q_2)$ is equal to the $t^0$
component of the product between $f(1/t)$ and $t/(1-tq_1)(1-tq_2)$
for $|tq_1|<1$ and $|tq_2|<1$. In this identification we have viewed
$1/(1-tq_1)(1-tq_2)$ as a Taylor series of variables $t,q_1,q_2$.
Now applying the same trick, we can identify
$\Delta^2(f)(q_1,q_2,q_3)$ as the $t^0$ component of the product
$f(1/t)t^2/(1-tq_1)(1-tq_2)(1-tq_3)$ for $|tq_i|<1$, $i=1,2,3$.
Extending this procedure, we have that in general, for $k\in
\naturals$, $\Delta^{k}(f)(q_1,\dots,q_{k})$ is the $t^0$
component of the product $f(1/t)t^k/(1-tq_1)\cdots(1-tq_k)$ for
$|tq_i|<1$, $i=1,\dots, k$. Finally, comparing this interpretation
of $\Delta^k(f)$, we conclude that
$\sum_{\alpha}\frac{c_\nu}{p_1^{m+n}}\frac{p_1^\alpha}{\alpha!}\partial_p^\alpha
a(x,0)$ is equal to $\Delta^{m+n}(a)$ evaluating at
\[
\underbrace{(x,p_1)\times \cdots \times
(x,p_1)}_{\Lambda_0+n_0+1}\times \underbrace{(x,-p_1)\times \cdots
\times (x-p_1)}_{\Lambda_1+n_1}.
\]
This is exactly the expression of the operator $A_\nu$, introduced
in Section \ref{subsec:bilinear-operators}, on the function $a$ with
respect to the variable $p_1$.

In summary, from the above expression about the symbol of the
operator $Op_k(a)\circ Op_k(b)$, we can quickly check property
(1)-(3) in Theorem \ref{thm:symbol-cal}.

\end{document}